%% file: Rigid-bodies-SQG.tex
\documentclass[11pt, twoside, a4paper]{article}
\usepackage{verbatim}
\input{commands}

\newcommand{\dx}{\,{\rm d} {x}}
\newcommand{\dy}{\,{\rm d} {y}}

\title{Local well-posedness for a moving rigid region in Surface Quasi-Geostrophic equations}
\author{Ludovic Godard-Cadillac$^1$ and Arnab Roy$^{2,3}$}
\date{\today}
\begin{document}

\maketitle
\begingroup
\renewcommand{\thefootnote}{}
\footnotetext{
\!\!\!\!\!$^1$ Bordeaux-INP, Institut de Mathématiques de Bordeaux (IMB), CNRS, Université de Bordeaux, 351 cours de la Libération, 33405 Talence (France)\newline
\indent$^2$ Basque Center for Applied Mathematics (BCAM), Alameda de Mazarredo 14, 48009 Bilbao, Spain\newline
\indent$^3$ Ikerbasque, Basque Foundation for Science, Plaza Euskadi 5, 48009 Bilbao, Bizkaia, Spain
}
\endgroup


\begin{abstract}
We introduce and analyze a class of Surface Quasi-Geostrophic (SQG) equations in the presence of moving rigid obstacles. The model is motivated both by vortex-wave type asymptotics for singular structures in active scalar equations and by geophysical phenomena exhibiting rigid-like coherent regions, such as cyclone eyes or long-lived atmospheric dust clouds.
We consider the critical SQG equation in a time-dependent exterior domain generated by a prescribed rigid motion and reconstruct the velocity through a nonlocal elliptic formulation adapted to impermeability constraints. The active scalar is assumed to remain constant inside the rigid region and in a neighborhood of its boundary, yielding a plateau structure compatible with the transport dynamics.
For a single moving obstacle, we establish local well-posedness of classical solutions in Sobolev spaces $H^k$, $k\geq 4$ together with uniqueness, local stability, and a blow-up criterion. The analysis relies on a reformulation in adapted coordinates reducing the problem to a fixed domain, combined with integral representations for the fractional elliptic operator, regularization procedures, and a nonlinear fixed-point argument. A central difficulty comes from the critical singularity of the SQG Biot–Savart kernel in the case $s=\frac{1}{2}$, for which the velocity reconstruction near the moving boundary requires commutator estimates. We further prove propagation of the plateau property and derive a priori estimates controlling both the support of the scalar gradient and the Sobolev norm of the solution.
This work provides, to our knowledge, the first well-posedness theory for SQG equations with moving rigid obstacles and constitutes a first step toward the rigorous derivation of point-vortex type dynamics from shrinking rigid bodies in SQG flows.
\end{abstract}



\section*{Introduction}
This article is devoted to the study of the Surface Quasi-Geostrophic equation (SQG) in the presence of a moving rigid obstacle. This equation is a standard model that arises in meteorology to describe a quasi-stratified fluid in a fast rotating frame. 
The motivations for a study of this equation in presence of a rigid obstacle have two aspects.

The first motivation comes from the point-vortex system for the Euler equations : it is known that the point-vortex model for the Euler equations can be obtained as a limit of small rigid obstacles~\cite{Glass_Munnier_Sueur_2018_Point_Vortex}. The obtained system can either be the standard point-vortex equation or the so-called vortex-wave system, studied for instance in~\cite{Lacave_Miot_2009,Lacave_Miot_2021}. 
This asymptotic result has never been generalized to (SQG) despite the strong proximity between these two equations~\cite{Cobb_Donati_Godard-Cadillac_2025}. The reason is that the theory of moving obstacles in (SQG) has not been developed yet despite the representation theory is well-understood~\cite{Abatangelo_2015,Djitte_Sueur_2023_Representation}.
Indeed, there exists a wide literature on obstacles in fluids driven by Euler equations~\cite{Glass_Lacave_Sueur_2014, Glass_Lacave_Sueur_2015,
Glass_Lacave_Munnier_Sueur_2019_Dynamics, 
Glass_Kolumban_Sueur_2020, 
Glass_Munnier_Sueur_2018_Point_Vortex, 
Glass_Sueur_2019_Dynamics,
Sueur_2012,
Glass_Sueur_Takahashi_2012,caggio2021measure,MR2719277} or by Navier-Stokes equations~\cite{serre1987chute, san2002global, MR2029294, nevcasova2021self, Planas_Sueur_2014, feireisl2023motion,  
Feireisl_Roy_Zarnescu_2023}. Concerning (SQG), only the case of motionless boundary have been studied~\cite{MR3595454, MR3956764, MR4076071, MR4093852, MR4822906}.
The main goal of this article is to provide one possible way to define SQG with moving obstacles.

On the other hand, there exists two meteorological phenomena for which such a rigidity behavior appear. The most famous one is the ``Eye of the cyclone'', a region in the center of the cyclone where the wind and the rain are very low, separate from the rest of the cyclone by the so-called ``Eyewall''.
In this very thin region separating the eye from the rest of the hurricane, the winds are the strongest in the entire cyclone and display macroscopic vertical motion, spiraling upward on the outer side of the wall and descending without spiraling on the inner side.
These vertical displacements break the stratification of the fluid in the wall and this explain this rigid-like behavior (this phenomenon being related to the Taylor-Proudman theorem). 
Although the eye is usually circular, polygonal eyewalls have sometimes been reported in small cyclones~\cite{Schubert_et_al_1999_Polygonal}. The link between rigidity of the central region and vorticity in Hurricane has been already widely studied from a physical point-of-view. See for instance~\cite{Dormy_Oruba_Emanuel_2024_Eye_Formation} and references therein. 

We also see the presence of rigid-like structure emerging in an other context : the massive Australian wildfires of January 2020 generated a persistent smoke-charged cloud in the atmosphere above the Pacific ocean. 
The small particles of ash and smoke in suspension in this cloud was absorbing the radiating energy of the sun and, from a mathematical point-of-view, this acted like a source term in the (SQG) equations. This cloud, before its dissipation at high altitudes, was moving like a rigid object and was surrounded by a massive vortex.~\cite{Khaykin_Legras_Bucci_2020}.

These observations provide strong motivation to better investigate the link between rigidity and vorticity in a context of Quasi-Geostrophic theory. In this article we consider a situation where the position of the rigid region is known at every moment of the dynamic and the active scalar is assumed to be constant inside and in a neighborhood of the rigid region. We mainly focus on the fluid dynamic and the active-scalar dynamic outside the rigid region.

The objective of this work is to construct a local well-posedness theory for the critical Surface Quasi-Geostrophic equation in the presence of moving rigid obstacles. The main analytical difficulty comes from the simultaneous presence of a nonlocal singular velocity law, a time-dependent fluid domain, and impermeability constraints induced by the rigid motion.
Our approach relies on a decomposition of the stream function into a singular SQG contribution and a correction term associated with the rigid motion. This decomposition allows us to reconstruct a divergence-free velocity field satisfying the impermeability condition at the moving boundary while preserving the transport structure of the SQG dynamics. In the critical case $s=\frac{1}{2}$, the associated Biot–Savart kernel is no longer locally integrable, requiring commutator estimates near the obstacle and careful control of the singular integral representation.
A central role in the analysis is played by the plateau assumption, namely the hypothesis that the active scalar remains constant in a neighborhood of the rigid region. This geometric property separates the singular support of the scalar gradient from the moving boundary and provides a mechanism allowing the nonlocal velocity reconstruction to remain regular despite the critical singularity of the kernel. The propagation of this plateau structure is therefore one of the key ingredients of the proof.
To overcome the difficulties associated with the moving domain, we introduce adapted variables transporting the equation onto a fixed reference configuration. In these coordinates, the SQG system becomes a nonlinear transport equation with time-dependent coefficients. We then derive integral representations for the fractional elliptic problem associated with the stream function and introduce suitable regularizations of the singular operators. The existence theory is obtained through an iterative scheme combined with compactness arguments and a nonlinear fixed-point procedure.
The proof of uniqueness and stability relies on energy estimates in Sobolev spaces adapted to the evolving geometry of the fluid region. In particular, we derive a priori estimates controlling both the support of the scalar gradient and the $H^k$ norm of the solution. These estimates lead to a continuation criterion showing that finite-time breakdown can only occur through blow-up of the Sobolev norm of the active scalar.

The paper is organized as follows.
In Section \ref{sec1}, we present the mathematical framework of the problem and discuss the connection with point-vortex dynamics for the two-dimensional Euler equations. We then introduce the SQG model with moving rigid obstacles and state the main well-posedness theorem.
Section \ref{sec2} is devoted to the reconstruction of the velocity field in the presence of moving obstacles. We introduce the decomposition of the stream function and establish the corresponding impermeability properties and integral representations for the fractional elliptic problem.
In Section \ref{sec3}, we focus on the single body case and we reformulate the equations in adapted variables on a fixed domain. We then construct suitable regularizations of the singular operators and establish several structural properties of the transport equation, including rearrangement and preservation of $L^p$-norms. This section culminates in the construction of an iterative scheme and the resolution of the associated fixed-point problem.
Section \ref{sec:a priori} contains the core a priori estimates of the paper. We first prove propagation of the plateau structure, showing that the active scalar remains constant in a neighborhood of the moving obstacle during the evolution. We then derive Sobolev estimates for the transported active scalar and establish bounds on the reconstructed velocity field in the critical SQG regime.
 Section \ref{sec5} is devoted to the proof of the main theorem. We establish local existence and uniqueness of classical solutions, derive the blow-up criterion, and prove the local stability estimate with respect to perturbations of the initial data. Finally, the appendix contains the proofs of several technical results and auxiliary estimates.

\section{Presentation of the problem and main result}\label{sec1}

\subsection{Point-vortices in 2D Euler flows}
We are interested in the zero-radius limit for rigid bodies immersed in a perfect Newtonian incompressible inviscid 2D flow. 
This problem has already been studied in the context of the Euler equations~\cite{Glass_Lacave_Munnier_Sueur_2019_Dynamics,Glass_Munnier_Sueur_2018_Point_Vortex, Glass_Sueur_2019_Dynamics}. We recall the 2D incompressible Euler equations equations in a domain  $\Omega$ that is not depending in time :
\begin{equation}
    \begin{cases}\displaystyle\frac{\partial u}{\partial t}+\di\big(u\otimes u\big)+\frac{1}{\rho}\nabla p=0,\vspace{0.2cm}\\
    \di u=0.\end{cases}
\end{equation}
In these equations, $u:[0,T)\times\Omega\to\R^2$ is the velocity of the fluid, $p:[0,T)\times\Omega\to\R$ is the pressure of the fluid and $\rho:[0,T)\times\Omega\to\R_+$ the density. The gradient of the pressure is the Lagrange multiplier associated to the incompressibility constraint $\di u=0$. The natural boundary condition for the fluid at $\partial\Omega$ is the so-called \textit{impermeability condition} : $u\cdot n=0$ where $n$ is the outward-pointing unit vector on $\partial\Omega$.

We also introduce the vorticity of the fluid $\omega:=\nabla^\perp\cdot u=\partial_2 u_1-\partial_1 u_2$. Here the notation $\perp$ denotes the counter clock-wise rotation of angle $\pi/2$ in the sense that $(x,y)^\perp=(y,-x)$ and $(\partial_1,\partial_2)^\perp=(\partial_2,-\partial_1)$. The vorticity satisfies the following transport evolution equation:
\begin{equation}
    \frac{\partial\omega}{\partial t}+u\cdot\nabla\omega=0.
\end{equation}
The velocity field can be deduced from the vorticity using the Biot-Savart law:
\begin{equation}
    u=-\nabla^\perp(-\Delta)^{-1}\omega\quad=:K_1\star\omega,
\end{equation}
The convolution kernel $K_1$ is called the \textit{Biot-savart kernel} and is obtained as the orthogonal gradient of the Green function of the Laplace operator. The Green function is a radial decreasing function given by
\begin{equation}
    G_1(x)=\frac{1}{2\pi}\log\bigg(\frac{1}{|x|}\bigg).
\end{equation}
The Biot-Savart Kernel is thus given by
\begin{equation}
    K_1(x)=\nabla^\perp G_1(x) = \frac{1}{2\pi}\frac{x^\perp}{|x|^2}
\end{equation}
It is proved in~\cite{Glass_Munnier_Sueur_2018_Point_Vortex} that in the presence of small obstacles surrounded by an irrotational fluid, the system converges as the diameter and the mass of the obstacles tend to $0$ towards the famous point-vortex system, which if $\Omega=\R^2$ writes: 
\begin{equation}
    \frac{\d{x}_i}{\d t}=\frac{1}{2\pi}\sum_{\substack{j=1\\j\neq i}}^Na_j\frac{(x_i-x_j)^\perp}{|x_i-x_j|^2}=\sum_{j=1}^{N}a_jK_1(x_i-x_j).
\end{equation}
In the above differential equation, we have $x_i:[0,T)\to\R^2$ the position of the $i^{th}$ point-vortex (which is the limit position of the $i^{th}$ obstacle) and $a_i\in\R$ the circulation of the fluid around the point-vortex. 
In the case where $\Omega\neq\R^2$, the point-vortex system contains extra terms for the interaction of the vortices with the boundary~\cite{Donati_2022, Donati_Godard-Cadillac_Iftimie_2024_Dynamics}.

This limit in~\cite{Glass_Munnier_Sueur_2018_Point_Vortex} gives another possible derivation of the point-vortex system obtained originally by taking formally a weighted sum of Dirac masses as initial datum : $\omega(0)=\sum_i^Na_i\delta_{x_i}.$ The other existing rigorous justifications of the point-vortex systems are done either in term of desingularization results~\cite{Marchioro_Pulvirenti_1993,Donati_Iftimie_2020_Long_Time} or by mean-field limits~\cite{GoodMan_Hou_Lowengrub_1990_Convergence}. In one case, the point-vortex system is obtained as limit of a sequence of smooth solutions. In the other case, the smooth solutions are obtained as mean-fields limits of point-vortices.

\subsection{Surface Quasi-Geostrophic vortices}
Another model of 2D perfect newtonian incompressible inviscid 2D flow coming from meteorology is the Surface Quasi-Geostrophic equations (SQG). The standard model for (SQG) is set in the whole plane as a non-linear transport equation, 
\begin{equation}
    \frac{\partial\omega}{\partial t}+v\cdot\nabla\omega=0,
\end{equation}
where $\omega:[0,T)\times\R^2\to\R$ is the active scalar\footnote{It is usual to denote by the letter $\theta$ the active scalar in the context of Quasi-Geostrophic theory since this quantity is physically homogeneous to a temperature. Nevertheless, in this article we need the notation $\theta$ for the angles of rotation of the rigid bodies . We then chose to use the notation usually devoted to the vorticity ; which is justified by the analogy between (SQG) and Euler 2D.} of the equation and $v:[0,T)\times\R^2\to\R^2$  the velocity field:
\begin{equation}\label{eq:Biot Savart frac}
    v=-\nabla^\perp(-\Delta)^{-s}\omega.
\end{equation}
The main difference with Euler 2D is the presence of the exponent $0<s<1$ in the Biot-Savart law (which is then sometimes called : the \textit{fractional} Biot-Savart law).
We recall that the fractional Laplace operator is defined by
\begin{equation}
    (-\Delta)^su(x)=p.v.\int_{\R^2}\frac{u(x)-u(y)}{|x-y|^{2+2s}}\d y
\end{equation}
If we formally take $s\to1^-$ then we recover the $2D$ Euler equation.

Similarly to the case of the full Laplace operator, the Green function of the fractional Laplace operator in the plane is a radial function given by:
\begin{equation}
    G_s(x)=\frac{c_s}{2(1-s)|x|^{2(1-s)}},\qquad\text{with }\quad c_s:=\frac{(1-s)\Gamma(1-s)}{2^{2s-1}\pi\Gamma(s)},
\end{equation}
where $\Gamma$ denotes the standard Gamma function. From this we can write the Biot-Savart Kernel associated to~\eqref{eq:Biot Savart frac}:
\begin{equation}
K_s(x):=\nabla^\perp G_s(x)=-c_s\frac{x^\perp}{|x|^{4-2s}}.
\end{equation}
Therefore, similarly to Euler 2D, the SQG equation is a transport equation on the active scalar (which here is $\omega$) with a velocity field that writes as a convolution of $\omega$ with an anti-symmetric function. Then we can consider, at least formally, the point-vortex system obtained we take a weighted sum of Dirac masses $\sum_{i=1}^Na_i\delta_{x_i}$ as initial datum.
The position of the Dirac masses then evolves in time according to the SQG point-vortex system:
\begin{equation}
    \frac{\d{x}_i}{\d t}=\sum_{\substack{j=1\\j\neq i}}^Na_jK_s(x_i-x_j).
\end{equation}
Concerning the rigorous justification of the point-vortex system for SQG, the singularity of the kernel makes the problem harder. 
For desingularization, only partial results exists in the general case~\cite{Rosenzweig_2020,Cavallaro_Garra_Marchioro_2021}. 
On the contrary, the mean-field limit of SQG point-vortices has been achieved by~\cite{Serfaty_2020_Mean_Field} but under a regularity assumption on the limit solution. For more details on the links between point-vortices in Euler and SQG, see for instance~\cite{Geldhauser_Romito_2020,Godard-Cadillac_vortex_2022}.\medskip

Concerning the physical meaning of the SQG equations when $s\neq 1/2$ : it is suggested in~\cite{Godard-Cadillac_2020_phd} that the case $s=1/2$ is only a special case where the vertical stratification of the fluid is of constant gradient. In the general case, the use of Caffarelli-Silvestre theory~\cite{Caffarelli_Sylvestre_2007} suggests that when $s\geq 1/2$, the equation can be seen as the Dirichlet-to-Neumann trace of a stratified fluid with non-constant gradient of density. For this article, we mainly focus on the case where $s=1/2$. We simplify the notations : 
\begin{equation}\label{eq:def:green}
    G(x):= G_\frac{1}{2}(x):=\frac{c}{|x|},\qquad\text{and}\qquad K(x):=K_\frac{1}{2}(x)=\nabla^\perp G(x)=-c\frac{x^\perp}{|x|^{3}}.
\end{equation}
We also denote by $\Lambda$ the fractional Laplace operator $-(-\Delta)^{1/2}.$

\subsection{Presentation of the main results}
Our ultimate goal is to extend to the SQG case the third existing justification of the point-vortex system, (vortices as a vanishing size of rigid obstacles), previously proved for the 2D Euler equations. In this article, we focus on the well-posedness of the model for fixed shape and size for the obstacles.
It is often a discussion whether the SQG equation with obstacles has clear physical meaning, unlike case without obstacles for which the physical modeling is clear~\cite{Godard-Cadillac_2020_phd}. As presented in introduction, we can find physical interest for the dynamics of hurricane eye or for massive dust clouds ; although this would be a very simplified model.
In our case, we want to analyze this system from a mathematical point-of-view only : the objective is to underpin the rough idea that vortices have ``rigidifying properties'' in the neighborhood of their centers (this point-of-view is the leading idea behind~\cite{Glass_Munnier_Sueur_2018_Point_Vortex}). 

More precisely, we consider the system of equations (presented with more details at the next section) for the Surface Quasi-Geostrophic equations with obstacles:
\begin{equation}\label{eq:THE SYSTEM}
    \left\{\begin{array}{ll}
             \displaystyle \frac{\partial\omega}{\partial t}(t,x)+v(t,x)\cdot\nabla\omega(t,x)=0,&\quad t\in[0,T),\quad x\in\cF(t),\vspace{0.2cm}\\
             \displaystyle \omega(t,x)=C_k& \quad  t\in[0,T),\quad x\in\cS_k(t),\vspace{0.2cm}\\
     \displaystyle    (-\Delta)^s\psi_1(t,x)=\omega(t,x)+(-\Delta)^s\psi_2(t,x)& \quad  t\in[0,T),\quad x\in\cF(t),\vspace{0.2cm}\\
     \displaystyle     \psi_2(t,x):=\sum_{\ell=1}^K\chi_\ell(t,x)\bigg(\frac{\d\theta_\ell}{\d t}(t)\,\frac{\big|x-h_\ell(t)\big|^2}{2}-x\cdot\frac{\d{h}_\ell}{\d t}(t)^\perp\bigg),&\quad  (t,x)\in[0,T)\times\R^2,\\
         v(t,x):=\nabla^\perp\psi_2(t,x)-\nabla^\perp\psi_1(t,x),&\quad (t,x)\in[0,T)\times\R^2,\\
         \omega(0,x)=\omega_0(x),&\quad x\in \R^2.
    \end{array}\right.
\end{equation}
The only unknown of this problem is the function $\omega:[0,T)\times\R^2\to\R$ which initial datum is denoted by $\omega_0:\R^2\to\R$. More precisely, it is the function $\omega$ in the fluid region. The datum of the problem are the sets $\cS_k(t)\subset\R^2$ and the constants $C_k$ which represent the active scalar inside the rigid region (assumed constant for simplicity). The set $\cS_k(t)$ represents the $k^\text{th}$ rigid obstacle with $h_k:[0,T)\to\R^2$ being its center and $\theta_k:[0,T)\to\R$ the angle of its orientation in the plane. $\cF(t)$ is the region where lays the fluid\footnote{Although this is not physically rigorous, we call ``fluid'' the part of the plane where we solve the surface quasi-geostrophic equation and ``solid'' the region submitted to the rigidity property presented in introduction.} and this region is defined by $\cF(t):=\R^2\setminus\cup_k\cS_k(t)$.

The velocity field for the fluid is $v:[0,T)\times\R^2\to\R^2$ ; it is not an unknown of the problem since the evolution equation is only on $\omega$. Nevertheless, it requires $\omega$, (with $h_k$, $\theta_k$ and $C_k$) to be computed. The unknown of the problem $\omega:[0,T)\times\R^2\to\R$ is the thermodynamical active scalar which corresponds physically to the \textit{relative potential temperature} (the deviation of potential temperature relatively to the mean potential temperature). 
The functions $\psi_1,\psi_2:[0,T)\times\R^2\to\R$ are the partial stream functions (the real stream function being the difference $\psi:=\psi_1-\psi_2$) and they are used to reconstruct the velocity field in a way that ensures the impermeability condition.
Finally, the cut-off functions $\chi_k$ are defined by
\begin{equation}
    \chi_k(t,x):=\chi\Big(\dist(x,\cS_k(t)\Big).
\end{equation}
where $\chi:\R\to[0,1]$ is a $\cC^\infty$ non-increasing function equal to $1$ on $\R_-$ and equal to $0$ on $[a,+\infty)$ for a fixed parameter $a>0$. We also recall that the distance of a point $x\in\R^2$ to a set $\cS\subseteq\R^2$ is defined by $\dist(x,\cS):=\inf_{y\in\cS}|x-y|$. \medskip

Only the active-scalar $\omega$ outside the rigid region is the \textit{unknown} of the problem. The dynamics of the rigid bodies and of the active scalar inside the rigid bodies are assumed to be known in this model. 
For the sake of simplicity, this article mainly focuses on the case where there is a single obstacle (meaning that we take $K=1$). The case $K>1$ will be studied in forthcoming works. 
We refer to this system throughout this article as \textit{SQG with moving obstacle}. At every time $t\geq0$ the domain $\R^2$ is then separated between two regions : the fluid that lays in $\cF(t)$ and the rigid region in $\cS(t):=\R^2\setminus\cF(t)$. We introduce the Banach space $\X^k_T$ for $k\in\N$ and $T>0$ as being the spaces associated to the following norms:
\begin{equation}\begin{split}
    \|\omega\|_{\X^k_T}:=\sup_{t\in[0,T)}\big\|\omega(t,\cdot)\big\|_{H^k(\cF(t))}.
    \end{split}
\end{equation}
We also work with the \textit{plateau property}, with consists in assuming that
\begin{equation}\label{eq:plateau}
    R(t):=\inf\Big\{|x-y|\;:\;x\in\cS(t),\quad y\in\cF(t)\cap\supp\big(\nabla\omega(t,\cdot)\big)\Big\}
\end{equation}
remains positive for all time.
\begin{theorem}\label{thrm}
We consider the system~\eqref{eq:THE SYSTEM} with $K=1$ and $C_1=0$. The dynamic of the rigid body is imposed by a given position in time $h\in\cC^\infty(\R_+)^2$ and a given orientation $\theta\in\cC^\infty(\R_+)$. Let $\omega_0\in H^k(\R^2)$ with $k\geq4$ be an initial datum for the active scalar of the fluid. Assume that $\omega_0$ satisfies the plateau hypothesis~\eqref{eq:plateau} at time $t=0$. Then,\medskip

    $(i)$ The Surface Quasi-Geostrophic equation with a moving obstacle admits a unique local classical solution $\omega$ in the space $\X_T^k$. This solution is classical in the sense: $\omega\in\cC^1([0,T)\times\R^2)$.\medskip
    
    $(ii)$ \emph{Blow-up criterion:} Assume that the maximal time of existence $T^\star$ is finite. Then:
    \begin{equation}
\int_0^{T^\star}\|\omega(t,\cdot)\|_{H^2(\cF(t))}\,\d t = +\infty.
    \end{equation}

    $(iii)$ \emph{Local stability:} Finally we have the following stability result : consider another solution $\omega^\dag$, we have
    \begin{equation}\label{stability1}
        \|\omega-\omega^\dag\|_{L^2(\cF(t))}\leq C\|\omega_0-\omega^\dag_0\|_{L^2(\cF(t))},
    \end{equation}
    where the constant $C$ depends smoothly on the initial datum $\omega_0$, on $T$ and on the motion of the obstacle). The above estimate holds only when the right-hand side of \eqref{stability1} is small enough (local stability).
\end{theorem}

For this article, we chose to focus on $s=1/2$ as it corresponds physically to a uniform stratification of the underlying 3D fluid model~\cite{Godard-Cadillac_2020_phd} and provide similar scalings as the more complicated models used in physical applications for meteorology.
From a mathematical point-of-view this case corresponds to the critical case where the kernel of the operator $\nabla^\perp(-\Delta)^{-s}$ ceases to be locally integrable and therefore the convolution with this kernel requires to be studied more carefully.
For these reasons we focused our work to the case $s=1/2$ but many of the features presented here can be extended to the sub-critical case $s>1/2$ without major modifications since this case is less singular. The super-critical case $s<1/2$ would require further modifications to control both the low and high frequencies with separate arguments, using decompositions similar to~\cite{Cobb_Donati_Godard-Cadillac_2025}.

\section{Surface Quasi-Geostrophic Equations with obstacles}\label{sec2}
This section is devoted to a general presentation of the system of equation and on how can moving obstacles be added in the standard Surface quasi-geostrophic equations.

\subsection{Dynamic of the rigid bodies}
The obstacles in this 2D fluid are denoted $\cS_k$ for $k=1,\dots,K$ and they form a 
family of 2-by-2 disjoint compact simply-connected subsets of $\R^2$ with a smooth boundary. These obstacles evolve in time so that the fluid lays in a domain depending on time:
\begin{equation}
    \cF(t):=\R^2\setminus\bigcup_{k=1}^K\cS_k(t)
\end{equation}
The motion of a rigid region in 2D can be described by the evolution of the position of its center of mass $h_k(t)\in\R^2$ and of its orientation described by an angle $\theta_k(t)\in\R$. For further use, we introduce the rotation matrix parameterized by an angle $\theta\in\R$:
\begin{equation}
    \bR_\theta:=\begin{pmatrix}
        \cos(\theta) & -\sin(\theta)\vspace{0.1cm}\\
        \sin(\theta) & \cos(\theta)
    \end{pmatrix}
\end{equation}
In particular, we recall that the rigid motion is entirely characterized by a translation and a rotation:
\begin{equation}
    \cS_k(t)=\bR_{\theta_k(t)}\big(\cS_k(0)-h_k(0)\big) +h_k(t).
\end{equation}
The evolution of the $k^{th}$ obstacle is given in this model. We expect more involved models to provide an equation on the evolution of $h_k$ and $\theta_k$. In the case of the Euler flows, the interaction of the obstacle with the fluid through the pressure at its surface. More precisely, in the Euler case, the center of mass and the orientation satisfy the following equations:
\begin{equation}\label{eq:evo h_k}
    m_k\frac{\d{h}_k}{\d t}(t)=\int_{\partial\cS_k(t)}p(t,x)\,n(x)\,\d\sigma(x),
\end{equation}
\begin{equation}\label{eq:evo vartheta_k}
    I_k\frac{\d{\theta}_k}{\d t}(t)=\int_{\partial\cS_k(t)}\big(x-h_k(t)\big)^\perp\cdot p(t,x)\,n(x)\,\d\sigma(x),
\end{equation}
where $p:[0,T)\times\R^2$ is the pressure of the fluid, $n(x)$ is the unit outward pointing vector on the boundary of the domain $\cS_k(t)$. We have $m_k$ the mass of the obstacle and $I_k$ its moment of inertia. 

In the case of SQG, it is not possible to consider such an evolution \eqref{eq:evo h_k}--\eqref{eq:evo vartheta_k} on the rigid region because of the non-local nature of the equation.
Deriving the evolution of the rigid region would require much more involved efforts of modeling ; this explains why we decided to fully focus on the study of the active scalar outside the obstacle in a general context.

\subsection{Reconstruction of the velocity in the presence of obstacles}\label{sec:obstacle}
It is also standard in a context of 2D Euler flow to impose the impermeability condition:
\begin{equation}\label{eq:impermeability}
    v(t,x)\cdot n(x)=\bigg(\frac{\d\theta_k}{\d t}(t)\big(x-h_k(t)\big)^\perp+\frac{\d{h}_k}{\d t}(t)\bigg)\cdot n(x),\qquad\forall\,x\in\partial\cS_k(t).
\end{equation}
The circulation of the fluid around a given rigid region $\cS_k(t)$ is preserved in time:
\begin{equation}
    \frac{\d}{\d t}\int_{\partial\cS_k(t)} v(t,x)\cdot n(x)^\perp\,\d\sigma(x)=0.
\end{equation}

In our case, there remains to explain how to compute the velocity field $v$ in such a way that this impermeability condition is satisfied at the boundary together with the fractional Biot-Savart law~\eqref{eq:Biot Savart frac} everywhere else in the fluid.
To achieve these two constraints in a self-contained manner, we introduce the notion of partial stream function as being any function $\psi_1:[0,T)\times\R^2\to\R$ that satisfies
\begin{equation}\label{eq:evo psi}
    \left\{\begin{array}{ll}
        \displaystyle\psi_1(t,x)=C_k& \qquad \text{for all}\;x\in\cS_k(t),\vspace{0.3cm}\\
         \Lambda\psi_1(t,x)=\omega(t,x)+\Lambda\psi_2(t,x)& \qquad \text{for all}\;\;x\in\cF(t),
    \end{array}\right.
\end{equation}
where $C_k$ is some fixed constant and where the corrective stream function $\psi_2$ is defined by
\begin{equation}\label{eq:def:inertial psi 2}
   \psi_2(t,x):=\sum_{\ell=1}^K\chi_\ell(t,x)\bigg(\frac{\d\theta_\ell}{\d t}(t)\,\frac{\big|x-h_\ell(t)\big|^2}{2}-x\cdot\frac{\d{h}_\ell}{\d t}(t)^\perp\bigg)
\end{equation}
with $\chi_k(t,\cdot)$ is a $\cC^\infty$ compactly-supported function equal to $1$ on a neighborhood of $\cS_k(t)$ and $0$ on a neighborhood of $\cS_\ell(t)$ for all 
$\ell\neq k$.
With such a definition, we can reconstruct the velocity field as being the vector field tangent to the stream lines with intensity equal to the norm of the gradient of $\psi_1$ plus the contribution given by the acceleration of the body. In short, we set:
\begin{equation}
    v(t,x):=\nabla^\perp\psi(t,x),
\end{equation}
where $\psi:=\psi_2-\psi_1$ is the total stream function.
In the equation above, it is \textit{a priori} not clear whether $\psi_1$ or its gradient are well-defined everywhere up to the boundary. Indeed, the only regularity for the velocity field one can expect, from solving~\eqref{eq:evo psi}, is Hölder-continuity (with an exponent depending on $s$). In this article, we restricted the analysis to simple case where the active scalar satisfies a \textit{plateau hypothesis}~\eqref{eq:plateau}, which in particular allows us to discard the singularity of $\nabla^\perp\psi(t,x)$ in the neighborhood of $\partial\cS_k$.
For a wider documentation on this lack of regularity problem, we refer to~\cite{Djitte_Sueur_2023_Representation, Abatangelo_2015}.

Despite this smoothness issue, we have the property that the boundary of $\cS_k$ is a level-set of the function $\psi_1$, and then at least formally its gradient is orthogonal to the boundary of $\cS_k(t)$:
\begin{equation}
    \nabla^\perp\psi_1(t,x)\cdot n(x)=0,\qquad\forall\,x\in\partial\cS_k(t).
\end{equation}
Therefore, for all $x\in\partial\cS_k$:
\begin{equation}\begin{split}
    v(t,x)\cdot n(x)&=n(x)\cdot\nabla^\perp\psi_2(t,x)
    \end{split}
\end{equation}
We now compute
\begin{equation}\begin{split}
    \nabla^\perp\psi_2(t,x)&=\Bigg[\sum_{\ell=1}^K\nabla\chi_\ell(t,x)\bigg(\frac{\d{\theta}_\ell}{\d t}(t)\,\frac{|x(t)-h_\ell(t)|^2}{2}-x\cdot\frac{\d{h}_\ell}{\d t}(t)^\perp\bigg)\\
    &\qquad+\sum_{\ell=1}^K\chi_\ell(t,x)\bigg(\frac{\d{\theta}_\ell}{\d t}(t)\,\big(x(t)-h_\ell(t)\big)-\frac{\d{h}_\ell}{\d t}(t)^\perp\bigg)\Bigg]^\perp\\
    &=\sum_{\ell=1}^K\nabla^\perp\chi_\ell(t,x)\bigg(\frac{\d{\theta}_\ell}{\d t}(t)\,\frac{|x(t)-h_\ell(t)|^2}{2}-x\cdot\frac{\d {h}_\ell}{\d t}(t)^\perp\bigg)\\
    &\qquad+\sum_{\ell=1}^K\chi_\ell(t,x)\bigg(\frac{\d{\theta}_\ell}{\d t}(t)\,\big(x(t)-h_\ell(t)\big)^\perp+\frac{\d{h}_\ell}{\d t}(t)\bigg).
    \end{split}
\end{equation}
If we now consider only the $x\in\partial\cS_k(t)$, then, given the definition of the $\chi_k$, the above formula reduces to
\begin{equation}
    \nabla^\perp\psi_2(t,x)=\frac{\d{\theta}_k}{\d t}(t)\,\big(x(t)-h_k(t)\big)^\perp+\frac{\d{h}_k}{\d t}(t),\qquad\forall\,x\in\partial\cS_k(t).
\end{equation}
Therefore, the velocity field $v$ satisfies the impermeability condition~\eqref{eq:impermeability} at the exterior boundary of the obstacles (at least formally in the sense of stream functions as precised above).

Concerning the fractional Biot-Savart law~\eqref{eq:Biot Savart frac} in the fluid domain, a direct computation shows that it is satisfied for all $x\in\cF(t)$:
\begin{equation}\begin{split}
    v(t,x)&=\nabla^\perp\psi_2(t,x)-\nabla^\perp\Lambda^{-1}\omega(t,x)-\Lambda^{-1}\Lambda\nabla^\perp\psi_2(t,x)\\&=-\nabla^\perp\Lambda^{-1}\omega(t,x).
    \end{split}
\end{equation}
We see in particular that the velocity field does not depend on the choice of the cut-off functions $\chi_k$ which only appear here as a technical tool to enforce the impermeability condition. 

\section{Adapted variables, regularization and fixed-point}\label{sec3}
We now focus exclusively on the case of a single obstacle ($K=1$) in the critical regime ($s=\frac12$) in order to prove Theorem~\ref{thrm}. Throughout the remainder of the paper, we drop the index $k=1,\dots,K$. Moreover, the value of the constant for $\psi_1$ in $\cS$ is chosen equal to $0$ (recall that $\psi_1$ is only defined up to an additive constant).

\subsection{Working in a fixed domain}
The first step in the proof of Theorem~\ref{thrm} is to exploit the fact that we are dealing with a single obstacle. This allows us to perform a change of variables and reformulate the problem on a time-independent domain. The price to pay is the appearance of additional terms in the equation, which are derived as follows.
\medskip

To begin with, observe that a material point $x(t)$ belonging to the rigid region $\mathcal S(t)$ satisfies the following evolution equation corresponding to a rigid motion:
 \begin{equation}
    \dot{x}(t)=\dot{h}(t)+\dot{\theta}(t)\big(x(t)-h(t)\big)^\perp.
\end{equation}
It is therefore natural to work in the frame attached to this rigid region, which is equivalent to define the change of variables  in lagrangian coordinates:
\begin{equation}
    y(t):=x(t)-h(t)-\int_0^t\dot{\theta}(t')\big(x(t')-h(t')\big)^\perp\d t'.
\end{equation}
If we define the rotation matrices
\begin{equation}
    e^{\theta\bR}=\begin{pmatrix}
        \cos\theta & -\sin\theta \\
        \sin\theta & \cos\theta\\
    \end{pmatrix}=\bR_\theta.
\end{equation}
This formulation eventually gives the following inversion formulas:
\begin{equation}
    x(t)=\bR_{\theta(t)}\,y(t)+h(t),\qquad\text{and}\qquad y(t)=\bR_{-\theta(t)}\Big(x(t)-h(t)\Big).
\end{equation}
Working again in Eulerian coordinates, we see that it is natural to define the new studied quantities after the following change of variable (which is equivalent to work in the frame associated to the solid):
\begin{equation}\begin{split}
    &\bullet\quad  \widetilde{\omega}(t,y):=\omega\Big(t,\bR_{\theta(t)}\,y+h(t)\Big)-2\dot{\theta}(t),\vspace{0.2cm}\\
    &\bullet\quad  \widetilde{\psi}(t,y):=\psi\Big(t,\bR_{\theta(t)}\,y+h(t)\Big)-\frac{\dot{\theta}(t)}{2}|y|^2+\dot{h}^\perp(t)\cdot\Big(\bR_{\theta(t)}y+h(t)\Big),\vspace{0.2cm}\\
    &\bullet\quad  \widetilde{v}:=\nabla^\perp\widetilde{\psi}=\bR_{-\theta(t)}\,v\Big(t,\bR_{\theta(t)}\,y+h(t)\Big)-\dot{\theta}(t)y^\perp-\bR_{-\theta(t)}\dot{h}(t).\\
    \end{split}
\end{equation}
We obtain the following evolution system:
\begin{equation}\label{eq:one steady obstacle}
    \left\{\begin{array}{ll}
         \displaystyle\frac{\partial\widetilde{\omega}}{\partial t}(t,y)+\nabla^\perp\widetilde{\psi}\cdot\nabla\widetilde{\omega}(t,y)=0,\vspace{0.2cm}&\\
         \widetilde{\psi_1}(t,y)=0&\text{for }y\in\widetilde{\cS},\vspace{0.2cm}\\
         \Lambda\displaystyle\widetilde{\psi_1}(t,y)=\widetilde{\omega}(t,y)+\Lambda\widetilde{\psi_2}(t,y)&\text{for }y\in\R^2\setminus\widetilde{\cS},\vspace{0.2cm}\\
         \widetilde{\psi}(t,y)=\widetilde{\psi_2}(t,y)-\widetilde{\psi_1}(t,y)-\widetilde{\varphi}(t,y),\vspace{0.2cm}\\
         \widetilde{\varphi}(t,y):=\dot{\theta}(t)\frac{|y|^2}{2}-\Big(\bR_{\theta(t)}y+h(t)\Big)\cdot\dot{h}^\perp(t),\vspace{0.2cm}\\
         \widetilde{\psi_2}(t,y):=\chi(\dist(y,\cS))\widetilde{\varphi}(t,y),\vspace{0.2cm}\\
         \widetilde{\omega}(0,y)=\omega_0(y).
    \end{array}\right.
\end{equation}
The evolution equation is well-defined almost everywhere in $\R^2$ despite a possible discontinuity of tangential component of the velocity field $\nabla^\perp\widetilde{\psi}$ across the boundary $\partial\widetilde{\cF}$.
For the definition of the function $\widetilde{\psi_2}$, it is defined in the inertial frame by~\eqref{eq:def:inertial psi 2}, which explains where comes the obtained formula above. We added this intermediate function $\widetilde{\varphi}$ as it simplifies the study later. We recall that $\chi:\R\in[0,1]$ is a $\cC^\infty$ non-increasing function equal to $1$ on $\R_-$ and equal to $0$ on $[a,+\infty)$ with $a>0$ fixed. We observe that in the fluid region neighboring the solid, the velocity field is given only by $\nabla^\perp\widetilde{\psi_1}$. We may then verify, at least formally, that the impermeability condition is preserved in the moving frame. Indeed, $\partial \mathcal S$ is a streamline of the flow, since it coincides with the boundary of the $0$-level set of $\widetilde{\psi_1}$.\medskip

For the rest of the article, we work in the new frame with a fixed domain independent of the time. To lighten the notation, we now drop the notation $\sim$ for the rest of the work.

\subsection{Integral representations and regularization}\label{sec:Green}
To obtain the announced well-posedness result, we first proceed to a regularization of the equation.
This regularization must be performed in such a way that we keep the structure of the PDE, which is a transport equation by a divergence-free vector field $\nabla^\perp\psi$ and a rigid obstacle.
To study the behavior of the function $\psi$, we have to study the representations formula for the inversion of the fractional Laplace operator in the plane. For the properties of these inversion formula, we heavily rely on~\cite{Abatangelo_2015, Djitte_Sueur_2023_Representation} for the study of equations of the form
\begin{equation}\label{eq:problem}
    \Lambda u=f\quad\text{in}\;\cF\qquad\text{and}\qquad u=0\quad\text{outside}\;\cF.
\end{equation}
Although their proofs are done in the case of smooth bounded domains $\cF$, The same estimates hold in the case where $\cF$ is the complement of a smooth bounded domain.
Indeed their proofs only involve estimates on integrals supported on the boundary of $\cF$. 

\begin{proposition}[Representation formulas for the fractional Laplace operator]\label{prop:representation}
    Let $\cF\subseteq\R^2$ be the complement of a smooth bounded set. There exists a unique solution $u\in L^2$ of the problem~\eqref{eq:problem}. This function is continuous and can be written as
    \begin{equation}
        u(x)=\int_{\cF}G_{\!\cF}(x,y)\,f(y)\,\d y,\qquad\text{for all }x\in\cF 
    \end{equation}
    where $G_{\!\cF}$ is the fractional Green function associated to $\cF$ defined by
    \begin{equation}
        \Lambda G_{\!\cF}(x,\cdot)=\delta_x,\quad\text{in }\cD'(\cF),\qquad\text{and}\qquad G(x,\cdot)=0\quad\text{in }\cF^c.
    \end{equation}
Moreover, this Green function can be decomposed as
    \begin{equation}
        G_{\!\cF}(x,y)\;=\;G(x-y)-H_{\!\cF}(x,y),
    \end{equation}
    where $G\equiv G_\frac{1}{2}$ is the Green function in the whole plane~\eqref{eq:def:green} and where $H_{\!\cF}$ is function that is smooth on $\cF$. Moreover, we have the following estimate:
    \begin{equation}\label{eq:estim smooth Green}
        \forall\,k,\ell\in\N,\quad\forall\,x,y\in\cF,\qquad \big|\nabla_{\!x}^k\nabla_{\!y}^\ell H_{\!\cF}(x,y)\big|\lesssim 1 +\frac{1}{\dist(x,\partial\cF)^{k+\frac{1}{2}}\,\dist(y,\partial\cF)^{\ell+\frac{1}{2}}},
    \end{equation}
    where the constant depends only on the geometry of $\partial\cF$.
\end{proposition}
This proposition is proved in appendix.

To regularize the system studied in this article, we make use of representation formula using green functions of the fractional Laplace operators. For this purpose, we introduce for a fixed regularization parameter $\delta\in(0,1]$, a regularized kernel $G_{\delta}$ as being a regularization of the kernels $G$ introduced at equation~\eqref{eq:def:green} for the whole plane. We focus here in the particular case $s=1/2$. This regularization of the kernel in the plane satisfies the following properties:\smallskip
\begin{itemize}
\item $\displaystyle x\in\R^2\,\longmapsto\,G_{\!\delta}(x)$ is a smooth radially symmetric non-increasing function.
    \item $\displaystyle\forall\,x\in\R^2\,s.t.\,|x|\in[\delta, +\infty),\quad G_{\!\delta}(x)=G(x).$
    \item $\displaystyle G_{\!\delta}(0)\leq2G_{\!\delta}(\delta u).\qquad$ ($u$ being any unit vector)
    \item $\displaystyle\|\nabla G_{\!\delta}\|_{L^\infty(\R)}\leq\|\nabla G\|_{L^\infty([2\delta,+\infty)}.$\smallskip
\end{itemize}
The other term that we have to regularize is the interaction term with the boundary. Concerning this other term, the function $H_{\!\cF}$, since it is smooth at positive distance from the neighborhood of the boundary, it is enough to regularize it close to $\partial\cF$. The best choice to deal with this regularization is to regularize the boundary term \textit{a posteriori} (after the convolution). More precisely, we consider the operator $ G_{\cF\!,\delta}$ defined on $L^2$ by
\begin{equation}\label{def:G F delta}
    G_{\cF\!,\delta}\,\phi(x):=\chi_\delta\biggl(\frac{|x|}{\delta^2}\biggr)\Big(1-\chi_\delta\big(\dist(x,\cS)\big)\Big)\int_{\cF}\Big(G_{\!\delta}(x-y)-H_\cF(x,y)\Big)\,\phi(y)\,\d y,
\end{equation}
where $\chi_\delta$ is a smooth non-increasing function on $\R$ equals to $1$ on $[0,\delta]$ and equals to $0$ on $[2\delta,+\infty)$ with $\delta>0$. 

Finally, we denote by $\R[X]$ the polynomials with real coefficients.

\begin{proposition}[Regularization of the Green function and Biot-Savart operator]\label{prop:calderon}
    Let $\delta>0$ and let the operator $G_{\cF\!,\delta}\,$ be defined by~\eqref{def:G F delta} on $L^2(\cF)$.\medskip

    $(i)$ \emph{(Schwartz regularity)} For all $k\in\N$, for all $p\in\R[X]$, there exists a constant $C_{k,p,\delta}$ such that for all $\phi\in L^2(\cF)$,
    \begin{equation}
        \big\|p\,\nabla^kG_{\cF\!,\delta}\,\phi\,\big\|_{L^2(\cF)}\leq C_{k,p,\delta}\,\|\phi\|_{L^2(\cF)},
    \end{equation}

    $(ii)$ \emph{(Calderon-Zygmund property)} There exists a constant $C$ depending only on $\cF$ such that for all $\phi\in L^2(\cF)$,
    \begin{equation}
        \|\nabla G_{\cF\!,\delta}\,\phi\|_{L^2(\cF)}\leq C\|\phi\|_{L^2(\cF)}.
    \end{equation}

    $(iii)$ \emph{(Convergence)} In the limit $\delta\to 0$, we have for all $\phi\in L^2(\cF)$,
    \begin{equation}
        \|G_{\cF\!,\delta}\,\phi-G_{\cF}\,\phi\|_{H^1(\cF)}\;\longrightarrow\;0.
    \end{equation}
\end{proposition}
The proof of this proposition is done in appendix.

From this proposition, we introduce the following fractional Biot-Savart regularized operator in the presence of boundaries:
\begin{equation}
    K_{\cF\!,\delta}:\phi\longmapsto\nabla^\perp G_{\!\cF\!,\delta}\,\phi.
\end{equation}
We also define the limit of the operators $K_{\cF\!,\delta}$ as being $K_\cF\phi\equiv K_{\cF\!,0}\phi:=\nabla^\perp(-\Delta)^{-\frac{1}{2}}\,\phi.$
This Biot-Savart operator is a Calderon-Zygmund operator of order $0$ that maps continuously $L^2(\cF)$ into itself and it satisfies (by Proposition~\ref{prop:calderon})
\begin{equation}
    \|K_{\cF\!,\delta}\,\phi\|_{L^2(\cF)}+\|K_{\cF}\,\phi\|_{L^2(\cF)}\leq C\|\phi\|_{L^2(\cF)}
\end{equation}
with a constant depending only on $\partial\cF$, and
\begin{equation}
        \|K_{\cF\!,\delta}\,\phi-K_{\cF}\,\phi\|_{L^2(\cF)}\;\longrightarrow\;0,\qquad\text{as }\delta\to 0.
\end{equation}

\subsection{Rearrangement property}
The transport equation on $\omega$ is an equation driven by a velocity field $\nabla^\perp\psi$ that is divergent free, as a consequence of the Schwarz theorem. The flow generated by this velocity field is then rearrangement flow, in the sense that it preserves the measure of the super-level sets of the functions. For this article, we define rearrangements as follows:
\begin{definition}[Rearrangements of functions]\label{def:rearrangements}
    Let $\phi$ and $\psi$ be two measurable functions on a domain $\Omega\subseteq\R^d$. We say that $\psi$ is a \textit{rearrangement} of $\phi$ iff:
    \begin{equation}
        \forall\,\eta>0,\qquad\cL^d\big\{\phi\geq\eta\big\}=\cL^d\big\{\psi\geq\eta\big\},\quad\text{and}\quad\cL^d\big\{-\phi\geq\eta\big\}=\cL^d\big\{-\psi\geq\eta\big\},
    \end{equation}
    where $\cL^d$ denotes the Lebesgue measure on $\R^d$ and where $\{\phi\geq\lambda\}$ is a short notation for the super-level sets of lambda defined as being $\{x\in\R^2:\phi(x)\geq\lambda\}$.    \smallskip
    
    We say that a family of functions $(\phi_t)_{t\in[0,T)}$ is a \textit{flow fo rearrangements} iff $\phi_t$ is a rearrangement of $\phi_0$ for all $t\geq0$.
\end{definition}
We can check that this relation of rearrangement is a relation of equivalence in the sense that it is a symmetric, reflexive and transitive relation. Moreover, it is a consequence of the construction of the integral that the rearrangements share the same value for the integral of compositions:

\begin{lemma}\label{lem:preservation of L p  norms}
    Let $\phi$ and $\psi$ be two measurable functions on a domain $\Omega \subset \mathbb{R}^d$ that are rearrangements of each other. Let $g:\R\to\R$ a function with bounded variations such that $g(0)=0$. If the function $g$ is non-negative, or if $g\circ\phi\in L^1$, then
    \begin{equation}
        \int_{\Omega}g\circ\phi=\int_{\Omega}g\circ\psi.
    \end{equation}
\end{lemma}
\begin{proof}
   For the sake of simplicity, we do the proof only in the case $g\geq0$, $\phi,\psi\geq0$. The general case follows by decomposing the functions in term of positive and negative parts. We also assume that $g\in\cC^1$. By the Fubini theorem, we have,
   \begin{equation}\begin{split}
       \int_{\Omega}g\circ\phi(x)\,\d x&=\int_\Omega g(0)+\int_\Omega\int_0^{\phi(x)}g'(s)\,\d s\,\d x \\&=\int_\Omega\int_0^{+\infty}\mathbbm{1}_{\{\phi\geq s\}}(x)\,g'(s)\,\d s\,\d x\\&=\int_0^{+\infty}\cL^d(\{\phi\geq s\})\,g'(s)\,\d s.
       \end{split}
   \end{equation}
   This last equality gives the preservation of $\int g\circ\phi$ by rearrangements. In the case where $g$ is not $ \cC^1$ but its derivative can be represented by a measure $\mu$ (regularity $BV$) the same proof holds replacing $g'(s)\,\d s$ by $\mu(\d s)$.
\end{proof}
A direct consequence of this lemma is the preservation of the $L^p$ norms by rearrangements:
\begin{equation}\label{eq:preservation of L p norms}
    \forall\,p\in[1,+\infty],\qquad\|\phi\|_{L^p}=\|\psi\|_{L^p}
\end{equation}
The case $p=+\infty$ can be directly proved using the definition of rearrangements. For the case $p<+\infty$, it is obtained by applying the previous lemma with $g(s):=|s|^p.$\smallskip

The interest for rearrangements appear in our work for the study of the transport equation since the transport equations associated to a divergent-free vector field generate a lagrangian flow that preserve the Lebesgue measure of the level-sets. More precisely, we have:

\begin{lemma}\label{lem:Liouville theorem}
    Let $\Omega$ be a domain of $\R^d$. Let $\phi:[0,T)\times\Omega\to\R$ and let $v:[0,T)\times\Omega\to\R^d$ be a divergent-free vector field. Assume that we have
    \begin{equation}
        \forall\,(t,x)\in[0,T)\times\Omega,\qquad\frac{\partial\phi}{\partial t} + v\cdot\nabla\phi=0,
    \end{equation}
    with impermeability condition $v(x)\cdot n(x)=0$ for $x\in\partial\Omega$ (with $n$ being a normal vector on the boundary).
    
    Then $t\mapsto\phi(t,\cdot)$ is a flow of rearrangements.
\end{lemma}
This classical result is often referred  as the \textit{Liouville's theorem} (see~\cite[Chap. 3.16]{Arnold_1989} for a wider presentation and a detailed proof).
The main consequence of this property for our work is that the $L^p$ norms are preserved by transport equations with divergent-free vector fields.

\subsection{Iterative scheme and first existence result}
The regularization presented in the previous section is not enough to ensure directly a first well-posedness result with estimates that does not degenerates depending on the parameters of the problem. 
To pass-by this difficulty, we study an iterative scheme that allows us to work with linear transport equations. 

To construct our iterative scheme, we assume that we already built functions $\omega_{\delta,i}$ for $i=0,\dots,n-1$ (with the same initial datum $\omega_0$) and we are interested in $\omega_{\delta,n}$ solution to 
\begin{equation}\label{eq:regularized iterative}
    \left\{\begin{array}{l}
         \displaystyle\frac{\partial{\omega_{\delta,n}}}{\partial t}(t,y)+\nabla^\perp{\psi}_{\delta,n-1}\cdot\nabla{\omega}_{\delta,n}(t,y)=0,\vspace{0.2cm}\\
         \displaystyle\psi_{1,\delta,n-1}(t,y):=G_{\cF\!,\delta}\,\Big({\omega}_{\delta,n-1}+\Lambda{\psi_2}\Big)(t,y),\vspace{0.2cm}\\
         \displaystyle{\varphi}(t,y):=\dot{\theta}(t)\frac{|y|^2}{2}-\Big(\bR_{\theta(t)}y+h(t)\Big)\cdot\dot{h}^\perp(t)\vspace{0.2cm}\\
         {\psi_2}(t,y):=\chi(\dist(y,\cS))\,{\varphi}(t,y),\vspace{0.2cm}\\
                  {\psi}_{\delta,n-1}:={\psi_2}-{\psi}_{1,\delta,n-1}-{\varphi},\vspace{0.2cm}\\
                  \omega_{\delta,n}(0,y)=\omega_0(y).
    \end{array}\right.
\end{equation}
where the notation $f(t,\cdot)$ stands for the function $y\mapsto f(t,y)$.

In the equations above, we can check that the stream function $\psi_{1,\delta,n-1}$ is $\cC^\infty$ on the whole $\R^2$ as a consequence of the definition of the regularization of the kernel $G_\cF$ and equal to $0$ in the rigid region.

We also need the following notation: $\displaystyle\quad
    H^\infty(\Omega)\;:=\;\bigcap_{k=0}^{+\infty} H^k(\Omega).$

\begin{lemma}[First existence result]\label{lem:first existence}
    Let $\delta >0$ and let $\omega_0\in H^\infty(\R^2)$, be initial data. Let $h\in\cC^\infty(\R_+)^2$ and $\theta\in\cC^\infty(\R_+)$.  Assume that the map $t\in\R_+\mapsto\omega_{\delta, n-1}(t,\cdot)$ is a flow of rearrangements (see Definition~\ref{def:rearrangements}) belonging to $\cC^\infty(\R_+;H^\infty(\R^2))$ and that it is constant on a neighborhood of the border $\partial\cS$ (plateau hypothesis). Assume moreover that, for all $k\in\N^\ast$,
    \begin{equation} \label{eq:control on the Sobolev norms}
        \|\omega_{\delta,n-1}(t,\cdot)\|_{H^k}\;\leq\;\|\omega_{\delta,n-1}(0,\cdot)\|_{H^k}\,f_k(t).
    \end{equation}
    for some smooth functions $f_k:\R_+\to\R_+$ that depend only on $t\mapsto h(t)$ and $t\mapsto\theta(t)$.\smallskip
    
    Then, there exists a unique global smooth solution $\omega_{\delta,n}$ to the system~\eqref{eq:regularized iterative} with the same initial datum and the same properties as for $\omega_{\delta, n-1}$ (smoothness, flow of rearrangement and plateau hypothesis).
\end{lemma}

\begin{proof} For the sake of readability, we remove the subscript $\delta$ for the writing of this proof. We consider the following ordinary differential equation :
    \begin{equation}\label{eq:ordinary}
        \frac{\d}{\d t}x(t)=\nabla^\perp\psi_{n-1}\big(t,x(t)\big),
    \end{equation}
  where ${\psi}_{n-1}={\psi_2}-{\psi}_{1,n-1}-{\varphi}$.  The estimate the propagation of regularity by the transport equation, we need to study the ordinary differential equation canonically associated, establish global well-posedness and study propagation of regularity. 
    This requires to estimate the regularity of the velocity field $f:=\nabla^\perp\psi_{n-1}$. 
    Concerning the regularity of the function $\nabla^\perp\psi_{1,n-1}$, it is a direct consequence of the regularizing properties of the operator $G_{\cF\!,\delta}\,$ (see section~\ref{sec:Green}). Indeed, we have
    \begin{equation}\label{bateau}
        \|\nabla^\perp\psi_{1,n-1}(t,\cdot)\|_{H^k}\,\leq\,C_k\|\omega_{n-1}(t,\cdot)+\Lambda\psi_2(t,\cdot)\|_{L^2}\leq\|\omega_{0}\|_{L^2}+\|\psi_2(t,\cdot)\|_{\dot{H}^{1}},
    \end{equation}
    where for the last inequality we used~\eqref{eq:preservation of L p norms} as $t\mapsto\omega_{n-1}(t,\cdot)$ is a flow of rearrangements by hypothesis.
    
    We now study the regularity in space $\nabla^\perp\psi_{2}$. We compute $        \nabla\psi_{2}:=A+B$ 
    with 
    \begin{equation}
A(t,y):=\nabla\chi(\dist(y,\cS))\bigg(\dot{\theta}(t)\frac{|y|^2}{2}-\Big(\bR_{\theta(t)}y+h(t)\Big)\cdot\dot{h}^\perp(t)\bigg)
    \end{equation}
    \begin{equation}
B(t,y):=\displaystyle\chi(\dist(y,\cS))\Big(\dot{\theta}(t)y-\bR_{\theta(t)}\cdot\dot{h}^\perp(t)\Big)
    \end{equation}
    Since we have $\nabla\chi(\dist(\cdot,\cS))$ compactly supported, we have $A(t,\cdot)$ globally lipschitz with respect to the $y$ variable (with a lipschitz constant locally bounded in time). 
    It is direct to check that $B(t,\cdot)$ is also globally lipschitz with respect to the $y$ variable.
    Concerning higher-order regularity for $\psi_2$, it is given by standard bootstrap arguments. Indeed, we observe that the cut-off $\chi$ is smooth and compactly supported, then the velocity field $\nabla^\perp\psi_2$ has all its higher order derivatives (both in space and time) that are continuous and compactly supported. 

    Concerning the regularity in space for the function $\nabla^\perp\varphi,$ we have
    \begin{equation}
\nabla^\perp\varphi(t,y)=\dot{\theta}(t)y-\bR_{\theta(t)}\cdot\dot{h}^\perp(t),
    \end{equation}
    which is clearly locally bounded in time and globally lipschitz in space.

    Standard theory on ordinary differential equations gives that~\eqref{eq:ordinary} is globally well-posed for all initial datum. 
    Indeed, the Cauchy-Lipschitz theorem ensures local well-posedness of the differential equation while the Gr\"{o}nwall Lemma ensures that the well-posedness is global when the velocity field is globally lipschitz with respect to its space variable (uniformly in time).

    We now denote by $\cS^{t}_{n-1}$ the flow associated to this equation starting from time $t=0$, meaning that $\cS^{t}_{n-1}$ satisfies
    \begin{equation}
        \cS^0_{n-1}x=x,\qquad\text{and}\qquad\frac{\d}{\d t}\cS^t_{n-1}x=\nabla^\perp\psi_{n-1}\Big(t,\cS^t_{n-1}x\Big).
    \end{equation}
We can then define $\omega_n$ with the following formula :  \begin{equation}\omega_n(t,\cS^t_{n-1}x):=\omega_0(x).\end{equation}
    We have $\omega_n$ well-defined with this formula since $\cS^{t}_{n-1}$ is a smooth diffeomorphism of the plane.
    It is direct to check that $\omega_n$ is solution to the transport equation in~\eqref{eq:regularized iterative}. The smoothness properties of $\nabla^\perp\psi_{n-1}$ ensures that $\omega_n$ also belongs to $\cC^\infty(\R_+;H^\infty(\R^2))$. Since $\nabla^\perp\psi_{n-1}$ is divergent-free by the Stokes theorem, we conclude that $\omega_n$ is a flow of rearrangements by Lemma~\ref{lem:Liouville theorem}.
    It is not clear \textit{a priori} that Property~\eqref{eq:control on the Sobolev norms} holds with functions that do not depend on $n\in\N.$ 
    Nevertheless, we observe that in the velocity field defining $\cS^{t}_{n-1}$, only the term that depends on $n\in\N$ is $\nabla^\perp\psi_{i,n-1}$.
    The Sobolev norms of this term are bounded independently on $n\in\N$ as given by~\eqref{bateau}.

    There remains to establish that the plateau hypothesis~\eqref{eq:plateau} that holds for $\omega_{n-1}$ (by hypothesis) still holds for $\omega_n$.
    We observe that we have $\nabla\psi_{1,n-1}$ that vanishes at the boundary of the fluid by construction (because of the cut-off properties of $G_{\cF\!,\delta}$ when $\delta >0$).
    Similarly, we have $\nabla(\psi_2-\varphi)$ vanishing at the boundary of the fluid because of the cut-off $\chi$. Since the functions are smooth, this implies that in the neighborhood of $\partial\cF$, we have
    \begin{equation}\label{Monterrey}
        \Big|\nabla\psi_{n-1}(t,x)\Big|\,\leq\,C_\delta\,\dist(x,\partial\cF)^2,
    \end{equation}
    where the constant $C_\delta$ only depends on $\delta>0$.
    We now consider a solution $t\mapsto x(t)$ of~\eqref{eq:ordinary} and we assume that at a given time $t$ we have $\dist(x(t),\partial\cS)<\rho(\partial\cS)/2$ where $\rho$ designates the smallest curvature radius of the boundary (we have $\rho>0$ by smoothness of $\cS$).
    As a consequence, for such a position $x$, we have
    \begin{equation}
        \Big\{y\in\partial\cS\;:\;|x-y|=\dist(x,\cS)\Big\},
    \end{equation}
    admits a unique element noted $\P_{\!\cS}x$ and which is called the projection of $x$ on $\cS$. This projector $\P_{\!\cS}$ is well-defined and smooth whenever we have $\dist(x,\cS)<\rho$. We now compute
    \begin{equation}\begin{split}
        \frac{\d}{\d t}\dist(x(t),\partial\cS)&=\frac{x(t)-\P_{\!\cS}x(t)}{|x(t)-\P_{\!\cS}x(t)|}\cdot\bigg(\frac{\d }{\d t}x(t)-\frac{\d}{\d t}\P_{\!\cS}x(t)\bigg)\\&=\frac{x(t)-\P_{\!\cS}x(t)}{|x(t)-\P_{\!\cS}x(t)|}\cdot\Big(I_d-\nabla_x\P_{\!\cS}(x(t))\Big)\frac{\d}{\d t}x(t)
        \\&=\frac{x(t)-\P_{\!\cS}x(t)}{|x(t)-\P_{\!\cS}x(t)|}\cdot\Big(I_d-\nabla_x\P_{\!\cS}(x(t))\Big)\nabla^\perp\psi_{n-1}\big(t,x(t)\big).
        \end{split}
    \end{equation}
    If we combine this with~\eqref{Monterrey}, we conclude that
    \begin{equation}
        \bigg|\frac{\d}{\d t}\dist(x(t),\partial\cS)\bigg|\;\lesssim\;\dist(x(t),\partial\cS)^2,
    \end{equation}
    where $\lesssim$ means that the inequality holds up to a multiplicative constant. If we now apply the Gr\"{o}nwall lemma to the function $t\mapsto   \dist(x(t),\partial\cS)$, we conclude that $x(t)$ does not reach $\partial\cS$ in finite time if it is initially at positive distance from $\partial\cS$. In particular, the plateau hypothesis holds for all time.
    \end{proof}

\subsection{Solving a fixed-point problem}
Having established this preliminary existence result, we may now pass to the limit as $n\to+\infty$ and prove the existence of a fixed point for the iterative scheme. This, in turn, yields a global solution to the regularized Surface Quasi-Geostrophic equation in the presence of a moving obstacle. More precisely, we prove in this section the following result:
\begin{proposition}\label{prop:regularized}
    Let $\delta>0$ and let $(\omega_{\delta,n})_{n\in\N}$ be the iterative sequence defined by Lemma~\ref{lem:first existence}. 

    Then this sequence convergences in $\cC^\infty(0,T;\cS(\R^2))$ for all $T>0$ towards $(\omega_{\delta})$ solution to the regularized Surface Quasi-Geostrophic equation with a moving obstacle:
\begin{equation}\label{eq:regularized}
    \left\{\begin{array}{l}
         \displaystyle\frac{\partial{\omega_{\delta}}}{\partial t}(t,y)+\nabla^\perp{\psi}_{\delta}\cdot\nabla{\omega}_{\delta}(t,y)=0,\vspace{0.2cm}\\
         \displaystyle{\psi_{1,\delta}}(t,y):=G_{\!\cF\!,\delta}\Big({\omega}_{\delta}+\Lambda{\psi_2}\Big)(t,y),\vspace{0.2cm}\\
         \displaystyle{\varphi}(t,y):=\dot{\theta}(t)\frac{|y|^2}{2}-\Big(\bR_{\theta(t)}y+h(t)\Big)\cdot\dot{h}^\perp(t)\vspace{0.2cm}\\
         {\psi_2}(t,y):=\chi(\dist(y,\cS)){\varphi}(t,y),\vspace{0.2cm}\\
                  {\psi}_{\delta}:={\psi_2}-\psi_{1,\delta}-{\varphi}.
    \end{array}\right.
\end{equation}
\end{proposition}
\begin{proof} Let $\delta>0$ and let $(\omega_n)$ be the sequence of functions given by Lemma~\ref{lem:first existence} (we omit the subscript $\delta$ for this proof to improve readability). Since we have $t\mapsto\omega_n(t,\cdot)$ that is a flow of rearrangements, we have (see Lemma~\ref{lem:preservation of L p  norms} for details)
\begin{equation}
\forall\,n\in\N,\quad\forall\,t\in\R_+,\qquad\|\omega_n(t,\cdot)\|_{L^p}=\|\omega_0\|_{L^p},\quad p\in[1,+\infty].
\end{equation} 
In particular, the sequence $(\omega_n)$ is weakly precompact in the space $L^\infty(\R_+;L^1\cap L^\infty(\R^2))$ so that we can assume (up to an omitted extraction) that it converges toward some function $\omega\in L^\infty(\R_+;L^1\cap L^\infty(\R^2)).$ 

Since we know that the higher-order derivatives are all controlled 
uniformly in $n\in\N$ as a consequence of~\eqref{eq:control on the Sobolev norms}, we conclude that the convergence actually holds strongly in $L^\infty(\R_+;H^k)$ for all $k\in\N$.
In particular, this allows us to pass to the limit in the equations. The proof that $t\mapsto\omega(t,\cdot)$ is a flow of rearrangements that satisfy the plateau hypothesis is similar from previous lemma.
\end{proof}

\section{A priori estimates}\label{sec:a priori}
Now that we have existence and uniqueness from the regularized SQG system with moving obstacle, we can deal with the limit $\delta\to0$. This limit requires a priori estimates that are performed in the same spirit as standard estimates for SQG \cite{Constantin_Majda_Tabak_1994, Chae_Constantin_Cordoba_Gancedo_Wu_2012, Cobb_Donati_Godard-Cadillac_2025}.\medskip


\subsection{A priori estimate on the support of $\nabla\omega$}
The first \textit{a priori} estimate consists in showing that the plateau hypothesis~\eqref{eq:plateau} is propagated in time on a small interval that can be chosen uniformly with respect to $\delta>0$ (provided that the behavior of $\omega$ is controlled). Following the ideas already developed in~\cite{Cobb_Donati_Godard-Cadillac_2025}, we introduce the following \textit{directional log-lipschitz semi-norm} for the study of the blow-up of the velocity field:
\begin{equation}
    \cN(v,\cS):=\sup_{\substack{x\in\cS+\cB(0,1)\\x\notin\overline{\cS}}}\;\sup_{y\in\P_{\!\cS} x}\;\frac{-(x-y)\cdot\big(v(x)-v(y)\big)}{|x-y|^2(1-\log(|x-y|))}
\end{equation}
where $\P_{\!\cS}$ designates the projection on the set $\cS$:
\begin{equation}
    \P_{\!\cS} x:=\Big\{y\in\overline{\cS}\;:\;|x-y|=\inf_{z\in\cS}|x-z|\Big\}.
\end{equation}
and where $\cS+\cB(0,1):=\{x+y:x\in\cS, y\in\cB(0,1)\}.$ 

We note that this quantity is always controlled by the standard log-lipschitz semi-norm on $\R^2\setminus\overline{\cS}$:
\begin{equation}\label{eq:directional log lip}
    \cN(v,\cS)\leq\|v\|_{LL(\R^2\setminus\overline{\cS})}:=\sup_{\substack{x,y\in\R^2\setminus\overline{\cS}\\|x-y|\leq 1}}\;\frac{|v(x)-v(y)|}{|x-y|(1-\log(|x-y|))}.
\end{equation}
Indeed, the adherence of the subset of $\R^2$ on which are done these two maximization are included one into the other.
In particular, if the quantity $\cN$ blows-up for a given velocity field, then the Log-lipschitz semi-norm of this velocity field also blows-up. 
Nevertheless, the divergence of $\cN$ carries much more information on the blow-up because the maximization process defining $\cN$ only involved vectors $x-y$ that are on the neighborhood of $\partial\cS$ and normal to this boundary ; but also a scalar product that takes large values only for velocity fields normal to the boundary (this is why it is called \textit{directional} semi-norm).

\begin{proposition}[Directional propagation of the plateau hypothesis]\label{prop:a priori R}
    {Let $\delta>0$, $\omega_\delta$ denote the solution given by Proposition~\ref{prop:regularized}, and define the associated radius $R_\delta(t)$ as in~\eqref{eq:plateau}. }Then we have
    \begin{equation}\label{eq:dir osgood}
        R_\delta(t)\geq R(0)\,\exp\bigg(-\exp\bigg(\int_0^t\cN\big((\nabla^\perp\psi_\delta(\tau,\cdot),\cS(\tau)\big)\,\d\tau\bigg)\bigg).
    \end{equation}
\end{proposition}
\begin{proof}
    Similarly as before, we make use of lagrangian coordinates and we assume for simplicity that $R(t)\leq 1$ (the other case being easy). We denote by $\tau\mapsto X(\tau)$ the trajectory of the fluid particle passing at position $x\in\R^2\setminus\overline{\cS}$ at a given time $t$ and let $y\in\partial\cS$. We have
    \begin{equation}
        \frac{\d}{\d\tau}\big|X(\tau)-y\big|_{\big|_{\tau=t}}=\frac{x-y}{|x-y|}\cdot\nabla^\perp\psi_\delta(t,x).
    \end{equation}
    if we now chose $y\in\P_{\!\cS} x\subset\partial\cS$, the expression above implies that
    \begin{equation}\begin{split}
        \frac{\d}{\d\tau}\dist(X(\tau);\cS)\;&=\;\min_{y\in\P_{\!\cS} x}\frac{x-y}{|x-y|}\cdot\nabla^\perp\psi_\delta(t,x)\\
        &=\;-\max_{y\in\P_{\!\cS} x}\frac{-(x-y)\cdot\nabla^\perp\psi_\delta(t,x)}{|x-y|}
        \end{split}
    \end{equation}
    We now recall that by construction the velocity field $\nabla^\perp\psi_\delta$ satisfies the impermeability condition. Since we have $x-y$ orthogonal to $\partial\cS$ when $y\in\P_{\!\cS} x$ then 
    \begin{equation}
        \forall y\in\P_{\!\cS}x,\qquad (x-y)\cdot\nabla^\perp\psi_\delta(t,y).
    \end{equation}
    Thus, if we assume that $x\in\cS+\cB(0,1)$,
    \begin{equation}
        \begin{split}
            \frac{\d}{\d\tau}\dist(X(\tau);\cS)\;&=\;-\max_{y\in\P_{\!\cS} x}\frac{-(x-y)\cdot\big(\nabla^\perp\psi_\delta(t,x)-\nabla^\perp\psi_\delta(t,y)\big)}{|x-y|}\\
            &\geq\;-\cN(\nabla^\perp\psi_\delta,\cS)\, |x-y|\big(1-\log|x-y|\big).
        \end{split}
    \end{equation}
    by definition of $\cN$. If we now chose $x$ such that $\dist(x,\cS)=R(t)$, then we have $|x-y|=R(t)$ for $y\in\P_{\!\cS}x.$ The inequality above becomes
    \begin{equation}
        \frac{\d}{\d t} R(t)\geq -\cN\big(\nabla^\perp\psi_\delta(t,\cdot),\cS(t)\big)\,R(t)\big(1-\log R(t)\big).
    \end{equation}
    If we integrate this inequality in time we are led to
    \begin{equation}
        \log(\log R(t)-1)-\log(\log R(0)-1)\geq-\int_0^t\cN\big(\nabla^\perp\psi_\delta(\tau,\cdot),\cS(\tau)\big)\,\d \tau.
    \end{equation}
    Exponentiating twice, we obtain the following Osgood-type bound \eqref{eq:dir osgood}.
    \end{proof}

\begin{corollary}[Propagation of the plateau hypothesis]\label{coro:a priori R}
    With the same hypothesis as in Proposition \ref{prop:a priori R}. We have 
    \begin{equation}
        R(t)\geq R(0)\exp\Big(-C\exp\Big(C\int_0^t\|\omega_\delta(\tau,\cdot)\|_{H^2}\,\d \tau\Big)\Big),
    \end{equation}
    where the constant $C$ depends only on the motion of the domain.
\end{corollary}
The proof consists mainly in establishing that
\begin{equation}
    \Big|\cN\big(\nabla^\perp\psi_\delta(t,\cdot),\cS(t)\big)\Big|\;\lesssim\;1+\|\omega(t,\cdot)\|_{H^2}.
\end{equation}
\begin{proof}
    It is a direct consequence of~\eqref{eq:directional log lip} that
    \begin{equation}\begin{split}
        \cN\Big(\nabla^\perp\psi_\delta(t,.),\cS(t)\Big)&\leq\|\nabla^\perp\psi_\delta(t,
        .)\|_{LL(\R^2\setminus\cS)}\\&\leq\|\nabla^\perp\psi_2(t,
        .)\|_{LL(\R^2\setminus\cS)}+\|\nabla^\perp\psi_{1,\delta}(t,
        .)\|_{LL(\R^2\setminus\cS)}+\|\nabla^\perp\varphi(t,
        .)\|_{LL(\R^2\setminus\cS)},
        \end{split}
    \end{equation}
    where the second inequality is given by the definition of $\psi_\delta$ at~\eqref{eq:regularized}. Since only $\psi_{1,\delta}$ depends on $\omega_\delta$, we are led to
    \begin{equation}\label{tekila}
        \exp\bigg(\int_0^t\cN\Big(\nabla^\perp\psi_\delta(\tau,.),\cS(t)\Big)\d\tau\bigg)\leq C\exp\bigg(\int_0^t\|\nabla^\perp\psi_{1,\delta}(\tau,.)\|_{LL}\d\tau\bigg),
    \end{equation}
    where the constant $C$ depends only on the dynamic of the rigid region and on the choice of the cut-off function $\chi$. Using now standard Sobolev embeddings for Log-lipschitz norms~\cite{Bahouri_Chemin_Danchin_2011} gives
    \begin{equation}
        \|\nabla^\perp\psi_{1,\delta}(\tau,.)\|_{LL}\leq C\|\nabla^\perp\psi_{1,\delta}(\tau,.)\|_{H^2(\R^2)}
    \end{equation}
    We now use the Calderon estimate given by Proposition~\ref{prop:calderon}-$(ii)$ to obtain
    \begin{equation}
        \|\nabla^\perp\psi_{1,\delta}(\tau,.)\|_{H^2(\R^2)}\leq\|\omega_\delta+\Lambda\psi_2\|_{H^2(\R^2)}
    \end{equation}
    Since $\psi_2$ depends only on the dynamic of the rigid motion, we can use the triangular inequality above and get
    \begin{equation}
        \exp\bigg(\int_0^t\|\nabla^\perp\psi_{1,\delta}(\tau,.)\|_{LL}\d\tau\bigg)\leq C\exp\bigg(\int_0^t\|\omega_\delta\|_{H^2(\R^2)}\bigg),
    \end{equation}
    where the constant $C$ depends only on the $H^2$ norm of $\Lambda\psi_2$. Plugging this estimate back into~\eqref{tekila} eventually leads to
    \begin{equation}
        \exp\bigg(\int_0^t\cN\Big(\nabla^\perp\psi_\delta(\tau,.),\cS(t)\Big)\d\tau\bigg)\leq C\exp\bigg(C\int_0^t\|\omega_\delta(\tau,.)\|_{H^2}\d\tau\bigg).
    \end{equation}
    We obtain the announced estimate using the previously Gr\"{o}nwall-Osgood estimate~\eqref{eq:dir osgood}.
\end{proof}

\subsection{A priori estimate on the $H^4$ norm of the fluid motion}

\begin{proposition}\label{prop:a priori H4}
    Let $\delta>0$ and let $\omega_\delta$ be the solution given by Proposition~\ref{prop:regularized}. Then we have
\begin{equation}
        \frac{\d}{\d t}\|\omega_\delta(t,\cdot)\|_{H^4 (\mathcal{F})}\;\lesssim\;\|\omega_\delta(t,\cdot)\|_{H^4 (\mathcal{F})}^3+\frac{1}{R(t)^5}.
    \end{equation}
\end{proposition}
\begin{proof}
Due to Lemma \ref{lem:Liouville theorem}, $t\mapsto\omega_{\delta}(t,\cdot)$ is a flow of rearrangements and we have preservation of $L^2$-norms $\|\omega_\delta(t,\cdot)\|_{L^2(\mathcal{F})}=\|\omega_\delta(0,\cdot)\|_{L^2(\mathcal{F})}$.

 Let us consider a multi-index $\alpha = (\alpha_1, \alpha_2)\in \mathbb{N}^2$ whose length is $k=|\alpha|= \alpha_1 + \alpha_2$. Let us set $D^{\alpha}=\partial_{1}^{\alpha_1}\partial_2^{\alpha_2}$ and compute 
\begin{equation}
D^{\alpha} (\nabla^\perp{\psi}_{\delta}\cdot\nabla{\omega}_{\delta}) = \nabla^\perp{\psi}_{\delta}\cdot \nabla D^{\alpha} {\omega}_{\delta} + \sum_{0<\beta\le\alpha}
\binom{\alpha}{\beta}
\big(D^{\beta}\nabla^\perp \psi_{\delta}\big)\cdot \nabla D^{\alpha-\beta}\omega_{\delta}.
\end{equation}
As $\operatorname{div} (\nabla^\perp{\psi}_{\delta})=0$, using equation \eqref{eq:regularized}$_1$ and integration by parts, we have the following identity
\begin{equation}\label{identity:H4}
\frac 12\frac{\d}{\d t}\int\limits_{\mathcal{F}} |D^{\alpha}\omega_\delta(t,\cdot)|^2 \dx = -\int\limits_{\mathcal{F}} D^{\alpha}\omega_\delta . \left[ \sum_{0<\beta\le\alpha}
\binom{\alpha}{\beta}
\big(D^{\beta}\nabla^\perp \psi_{\delta}\big)\cdot \nabla D^{\alpha-\beta}\omega_{\delta}\right] \dx.
\end{equation}
In order to estimate the right-hand side of \eqref{identity:H4}, let us consider
\begin{equation}
I_{\alpha,\beta}
:=
\int_{\mathcal F_0}
(D^{\beta}\nabla^\perp \psi_{\delta})\cdot \nabla D^{\alpha-\beta}\omega_\delta \,
D^\alpha \omega_\delta \dx,
\qquad 0<\beta\le\alpha.
\end{equation}
Fix $|\alpha|=4$ and we distinguish cases according to $|\beta|$. Recall that
\begin{equation}
\psi_\delta
=
\psi_2
-
G_{\cF,\delta}\big(\omega_\delta + (-\Delta)^{1/2}\psi_2\big)
-
\varphi,
\end{equation}
and we aim at estimating $
\|D^\beta \nabla^\perp \psi_\delta\|_{L^2(\cF_0)}, \quad |\beta|\le 4$. We set $\phi := \omega_\delta + (-\Delta)^{1/2}\psi_2$ and write
$$D^\beta \nabla^\perp \psi_\delta
=
D^\beta \nabla^\perp \psi_2
-
D^\beta \nabla^\perp G_{\cF,\delta}\phi
-
D^\beta \nabla^\perp \varphi.$$
By construction, ${\varphi}(t,y):=\dot{\theta}(t)\frac{|y|^2}{2}-\Big(\bR_{\theta(t)}y+h(t)\Big)\cdot\dot{h}^\perp(t)$, so, $\varphi$ is a polynomial of degree $2$ and $\psi_2=\chi(\dist(\cdot,\cS))\varphi$ is smooth and compactly supported near the boundary. Hence,
\begin{equation}\label{der:phi}
\|D^\beta \nabla^\perp \psi_2\|_{L^2(\cF_0)}
+
\|D^\beta \nabla^\perp \varphi\|_{L^2(\cF_0)}
\le C,
\qquad |\beta|\le 4.
\end{equation}
We denote $K_{\cF,\delta}=\nabla^\perp G_{\cF,\delta}$. By Proposition~\ref{prop:calderon}, we have
\begin{equation}
\| K_{\cF,\delta}\phi\|_{L^2(\cF_0)}
\le C_{\beta,\delta}\|\phi\|_{L^2(\cF_0)}.
\end{equation}
Moreover,
\begin{equation}
\|\phi\|_{L^2(\cF)}
\le
\|\omega_\delta\|_{L^2(\cF)}
+
\|(-\Delta)^{1/2}\psi_2\|_{L^2(\cF)}
\le
\|\omega_\delta\|_{L^2(\cF)} + C.
\end{equation}
We assume the \emph{plateau property}: there exists $R(t)>0$ such that
\begin{equation}\label{eq:plateau_recall}
    R(t):=\inf\Big\{|x-y|\;:\;x\in\cS(t),\quad y\in\supp\big(\nabla\omega(t,\cdot)\big)\Big\}>0.
\end{equation}
In particular, $\omega(t,\cdot)$ is locally constant in a neighborhood of $\partial\cS(t)$. We exploit the separation property \eqref{eq:plateau_recall}. Let $x$ belong to the region where $\psi_2$ is supported (i.e. a neighborhood of $\partial\cS$), and let $y\in\supp(\nabla\omega)$. Then $|x-y|\ge R(t)$.
Therefore, for all $k\ge 1$, due to decomposition in Proposition \ref{prop:representation} and estimate \eqref{eq:estim smooth Green}, the kernel satisfies
\begin{equation}
|\nabla^k (G_\delta(x-y)-H_{\cF}(x,y))|
\;\le\;
\frac{C}{|x-y|^{k+1}}
\;\le\;
\frac{C}{R(t)^{k+1}}.
\end{equation}

We split $\omega_\delta$ into a near-field and far-field part. In a neighborhood of the boundary, $\omega_\delta$ is constant thanks to the plateau property, hence the corresponding contribution vanishes at leading order by cancellation of Calder\'on--Zygmund kernels. The remaining contribution involves only interactions at distance at least $R(t)$, and we obtain
\begin{equation}
|D^\beta K_{\cF,\delta}\omega_\delta(x)|
\le
\frac{C}{R(t)^{|\beta|+1}}\|\omega_\delta\|_{L^1(\cF_0)}.
\end{equation}
Collecting the above estimates, we obtain for all $|\beta|\le 4$:
\begin{equation}
\label{eq:final_plateau_estimate}
\|D^\beta \nabla^\perp \psi_\delta\|_{L^2(\cF_0)}
\;\le\;
C
+
C_{\beta,\delta}
\left(
\|\omega_\delta\|_{L^2(\cF_0)}
+
\frac{1}{R(t)^{|\beta|+1}}
\right).
\end{equation}

\medskip

\noindent
\textbf{Case 1: $|\beta|=1$.}
Then $|\alpha-\beta|=3$. Using H\"{o}lder's inequality and the Sobolev embedding, we have
\begin{align}
|I_{\alpha,\beta}|
&\le
\|D^\beta \nabla^\perp \psi_\delta\|_{L^\infty(\mathcal{F})}
\|\nabla D^{\alpha-\beta}\omega_\delta\|_{L^2(\mathcal{F})}
\|D^\alpha \omega_\delta\|_{L^2(\mathcal{F})} \\
&\lesssim
\|\nabla^\perp \psi_\delta\|_{H^3 (\mathcal{F})}
\|\omega_\delta\|_{H^4 (\mathcal{F})}
\|D^\alpha \omega_\delta\|_{L^2 (\mathcal{F})}.
\end{align}

\medskip

\noindent
\textbf{Case 2: $|\beta|=2$.}
Then $|\alpha-\beta|=2$. We use the Sobolev embedding to obtain
\begin{align}
|I_{\alpha,\beta}|
&\le
\|D^\beta \nabla^\perp \psi_\delta\|_{L^4(\mathcal{F})}
\|\nabla D^{\alpha-\beta}\omega_\delta\|_{L^4(\mathcal{F})}
\|D^\alpha \omega_\delta\|_{L^2(\mathcal{F})} \\
&\lesssim
\|\nabla^\perp \psi_\delta\|_{H^3(\mathcal{F})}
\|\omega_\delta\|_{H^4(\mathcal{F})}
\|D^\alpha \omega_\delta\|_{L^2(\mathcal{F})}.
\end{align}

\medskip

\noindent
\textbf{Case 3: $|\beta|=3$.}
Then $|\alpha-\beta|=1$, and we estimate
\begin{align}
|I_{\alpha,\beta}|
&\le
\|D^\beta \nabla^\perp \psi_\delta\|_{L^2(\mathcal{F})}
\|\nabla D^{\alpha-\beta}\omega_\delta\|_{L^\infty(\mathcal{F})}
\|D^\alpha \omega_\delta\|_{L^2(\mathcal{F})} \\
&\lesssim
\|\nabla^\perp \psi_\delta\|_{H^3(\mathcal{F})}
\|\omega_\delta\|_{H^4(\mathcal{F})}
\|D^\alpha \omega_\delta\|_{L^2(\mathcal{F})}.
\end{align}
Now using the estimate \eqref{eq:final_plateau_estimate} and from the above three cases, we obtain for all $|\beta|\le 3$:
\begin{equation}
|I_{\alpha,\beta}| \;\lesssim\;\|\omega_\delta(t,\cdot)\|_{H^4 (\mathcal{F})}^3+\frac{1}{R(t)^{|\beta|+1}}.
\end{equation}
\medskip

\noindent
\textbf{Case 4: $|\beta|=4$.}
Then $|\alpha-\beta|=0$ and we have the most delicate term in the high-order energy estimate:
\begin{equation}
I_{4,4}
:=
\int_{\mathcal F_0}
(D^{4}\nabla^\perp \psi_{\delta})\cdot \nabla \omega_\delta \,
D^4 \omega_\delta \dx.
\end{equation}
We can write  
\begin{equation}\label{decompose:nablapsi}
D^4 \nabla^\perp \psi_\delta
=
D^4 \nabla^\perp \psi_2
-
D^4 \nabla^\perp (G_\delta-H_{\cF})\phi
-
D^4 \nabla^\perp \varphi. 
\end{equation}
As $\phi := \omega_\delta + (-\Delta)^{1/2}\psi_2$, we want to analyze the term $I= \int_{\mathcal F_0}
(D^{4}\nabla^\perp \omega_{\delta})\cdot \nabla \omega_\delta \,
D^4 \omega_\delta \dx.$ Due to \eqref{eq:plateau_recall}, we have $I= \int_{\mathbb{R}^2}
(D^{4}\nabla^\perp \omega_{\delta})\cdot \nabla \omega_\delta \,
D^4 \omega_\delta \dx.$ Let us set $A_1=-\partial_2 G_{\delta}$ and  
$A_2=\partial_1 G_{\delta}$. We can write
\begin{align}
I=\sum\limits_{j=1}^2 \int_{\mathbb{R}^2} D^4(A_{j}\omega_{\delta})\partial_{j} \omega_{\delta} D^4 \omega_{\delta} \dx &= -\sum\limits_{j=1}^2 \int_{\mathbb{R}^2}A_{j}(D^4 \omega_{\delta}\partial_{j} \omega_{\delta} ) D^4 \omega_{\delta} \dx \\ &=\frac{1}{2} \sum\limits_{j=1}^2 \int_{\mathbb{R}^2} D^4 \omega_{\delta} [\partial_j \omega_{\delta}, A_j ] D^4 \omega_{\delta} \dx.
\end{align}
Now we use the similar commutator estimate as in \cite[Estimate (3.6) in Lemma 3.3]{Cobb_Donati_Godard-Cadillac_2025} to conclude that $|I| \;\lesssim\;\|\omega_\delta(t,\cdot)\|_{H^4 (\mathcal{F})}^3$. Due to the decomposition \eqref{decompose:nablapsi}, estimate \eqref{eq:estim smooth Green} of $H_{\cF}$ and estimate \eqref{der:phi}, we obtain  
\begin{equation}
|I_{4,4}| \;\lesssim\;\|\omega_\delta(t,\cdot)\|_{H^4 (\mathcal{F})}^3+\frac{1}{R(t)^{5}}.
\end{equation}
\end{proof}

\section{Proof of Theorem~\ref{thrm}}\label{sec5}

\subsection{Local existence of a solution}

From the a priori estimates previously established at Section~\ref{sec:a priori}, it is standard to establish existence of a solution by compactness arguments as $\delta\to0$. For the sake of completeness, we provide a detailed proof for the existence result.

\begin{lemma}[Extraction of a converging sequence]\label{lem:extraction}
    Let $\omega_\delta$ be the solution given by Proposition~\ref{prop:regularized}. Then  we have (up to an extraction) as $\delta\to0$:
    \begin{equation}
        \omega_\delta\;\longrightarrow\;\omega\;\qquad\text{strongly in }\;\cC^0(0,T;H^3_{loc}(\R^2)),\quad\text{and weakly-}\star\text{ in }\;\cC^0(0,T;H^4_w(\R^2)),
    \end{equation}
    where $H^4_w$ designates the space $H^4$ endowed with the weak topology.
\end{lemma}

\begin{proof}
    The uniform bounds provided by the \textit{a priori} estimates (Proposition~\ref{prop:a priori R} and its corollary~\ref{coro:a priori R} for the estimate on the plateau hypothesis and Proposition~\ref{prop:a priori H4} for the $H^4$ estimate) provide the compactness with respect to the space variable (weakly in $H^4$ and strongly in $H^3_{loc}$ by the Rellich-Kondrachov theorem). Concerning the time variable, it is studied using the equation. More precisely, 
    \begin{equation}
\|\partial_t\omega_\delta\|_{H^\ell(\R^2)}=\|\nabla^\perp\psi_\delta\cdot\nabla\omega_\delta\|_{H^\ell(\R^2)}
    \end{equation}
    using a Sobolev embedding, we have
    \begin{equation}
    \|\partial_t\omega_\delta(t,\cdot)\|_{L^\infty}\leq C  \|\partial_t\omega_\delta(t,\cdot)\|_{H^2}=\|\nabla^\perp\psi_\delta\cdot\nabla\omega_\delta(t,\cdot)\|_{H^2}.
    \end{equation}
    Therefore,
    \begin{equation}
        \|\partial_t\omega_\delta(t,\cdot)\|_{L^\infty}\lesssim\|\nabla\psi_\delta\|_{W^{2,\infty}}\|\omega_\delta(t,\cdot)\|_{H^3}\lesssim\|\nabla\psi_\delta\|_{H^4}\|\omega_\delta(t,\cdot)\|_{H^3},
    \end{equation}
    where the last inequality is given by the Sobolev embedding $H^2\subset L^\infty$.
    We now use the Calderon-Zygmund property given by Proposition~\ref{prop:calderon}-$(ii)$. We get
        \begin{equation}
        \|\partial_t\omega_\delta(t,\cdot)\|_{L^\infty}\lesssim\|\omega_\delta(t,\cdot)+\Lambda\psi_2(t,\cdot)\|_{H^4}\|\omega_\delta(t,\cdot)\|_{H^3}\lesssim 1+\|\omega_\delta(t,\cdot)\|_{H^4}^2,
    \end{equation}
    where the omitted multiplicative constant may depend on $\|\Lambda\psi_2\|_{H^4}$. This bound on the derivative in time provides lipschitz regularity in time and hence equi-continuity. Since we already known that for all $t\in[0,T)$ fixed, the family $\{\omega_\delta(t)\}_{\delta>0}$ is relatively compact in $H^3(K)$ for any compact $K\subset\subset\R^2.$
    Therefore, by the Ascoli-Arzela theorem, we have convergence in $\cC^0(0,T, H^3(K))$ up to an omitted extraction.
    
\end{proof}

Now that we obtained a limit object $\omega\in \cC^0(0,T; H^4(\R^2))$, we have to check if it satisfies the equation.

\begin{lemma}[Equation satisfied by the limit]\label{lem:existence}
    For a given initial value $\omega_0\in H^4(\R^2)$ such that $\omega_0\equiv0$ on $\cS(0)$ with plateau hypothesis~\eqref{eq:plateau}, there exists a local solution $\omega\in \cC^0(0,T; H^4_w(\R^2)) \cap \cC^1([0,T)\times\R^2)$  to the studied system~\eqref{eq:one steady obstacle}.
\end{lemma}

This lemma eventually concludes the proof of Theorem~\ref{thrm}-$(i)$ for the existence of a solution.

\begin{proof}
First we write by linearity
\begin{equation}
    \psi_{1,\delta}=G_{\!\cF\!,\delta}\big(\omega_\delta-\omega\big)+G_{\!\cF\!,\delta}\big(\omega+\Lambda\psi_2\big).
\end{equation}
Using the Calderon-Zygmund property stated by Proposition~\ref{prop:calderon}-$(ii)$, we have (since $G_{\!\cF,\!\delta}$ commutes with the derivation operators)
\begin{equation}\begin{split}
    \big\|\nabla^\perp G_{\!\cF\!,\delta}\big(\omega_\delta&-\omega\big)\big\|_{H^3(\Omega)}^2=\sum_{\ell=0}^3\big\|\nabla^\perp G_{\!\cF\!,\delta}\nabla^\ell\big(\omega_\delta-\omega\big)\big\|_{L^2(\Omega)}^2\\&\lesssim\sum_{\ell=0}^3\big\|\nabla^\ell\big(\omega_\delta-\omega\big)\big\|_{L^2(\Omega)}^2=\|\omega_\delta-\omega\|_{H^3(\Omega)}^2\quad\longrightarrow0,
    \end{split}
\end{equation}
for all $\Omega\subseteq\R^2$ bounded. 
On the other hand, Proposition~\ref{prop:calderon}-$(iii)$ gives
\begin{equation}
    G_{\!\cF\!,\delta}\big(\omega-\Lambda\psi_2\big)\;\longrightarrow\; G_{\!\cF}\big(\omega-\Lambda\psi_2\big)\;=:\psi_1
\end{equation}
Therefore, we have the following equality everywhere:
\begin{equation}\label{eq:transport}
    \frac{\partial\omega}{\partial t}+\nabla^\perp\psi\cdot\nabla\omega = 0,
\end{equation}
where $\psi:=\psi_2-\psi_1-\varphi$. 
\smallskip

We now investigate regularity with respect to the time variable. 
If we differentiate the transport equation~\eqref{eq:transport} with respect to time, the chain-rule eventually leads to the following estimate (using the Cauchy-Schwarz inequality):
\begin{equation}\label{point de depart}
    \|\partial^2_t\omega\|_{L^2} \leq\|\nabla\partial_t\psi\|_{L^4}\|\nabla\omega\|_{L^4}+\|\nabla\psi\|_{L^\infty}\|\nabla\partial_t\omega\|_{L^2}.
\end{equation}

Concerning the first term, we use a Gagliardo-Nirenberg inequality:
\begin{equation}\begin{split}
\|\nabla\partial_t\psi\|_{L^4}\|\nabla\omega\|_{L^4}&\lesssim\|\nabla\partial_t\psi\|_{L^2}^\frac{1}{2}\|\nabla^2\partial_t\psi\|^\frac{1}{2}_{L^2}\|\nabla\omega\|^\frac{1}{2}_{L^2}\|\nabla^2\omega\|^\frac{1}{2}_{L^2}\\
&\lesssim\|\nabla\partial_t\psi\|_{H^1}^2+\|\nabla\omega\|_{H^1}^2.
    \end{split}
\end{equation}
We replace $\psi$ by its definition and commute $G_{\!\cF\!,\delta}$ with $\partial_t$:
\begin{equation}\begin{split}
\|\nabla\partial_t\psi\|_{L^4}\|\nabla\omega\|_{L^4}&\lesssim\|\nabla G_{\!\cF}\,\partial_t(\omega+\Lambda\psi_2)\|_{H^1}^2+\|\nabla\omega\|_{H^1}^2\\
    &\lesssim1+\|\partial_t\omega\|_{H^1}^2+\|\nabla\omega\|_{H^1}^2
    \end{split}
\end{equation}
where for the last inequality we used Proposition~\ref{prop:calderon}-$(ii)$ (the omitted multiplicative constant depends on $\Lambda \psi_2$). We now replace $\partial_t\omega$ by the transport equation:
\begin{equation}\begin{split}
\|\nabla\partial_t\psi\|_{L^4}\|\nabla\omega\|_{L^4}\lesssim1+\|\nabla^\perp\psi\cdot\nabla\omega\|_{L^2}^2+\|\nabla(\nabla^\perp\psi\cdot\nabla\omega)\|_{L^2}^2+\|\nabla\omega\|_{H^1}^2
\end{split}
\end{equation}
We now estimate, using the Cauchy-Schwarz and the Gagliardo-Nirenberg inequality,
\begin{equation}
\|\nabla^\perp\psi\cdot\nabla\omega\|_{L^2}^2\lesssim\|\nabla^\perp\psi\|_{L^4}^2\|\nabla\omega\|_{L^4}^2\lesssim\|\nabla^\perp\psi\|_{H^1}^2\|\nabla\omega\|_{H^1}^2
\lesssim1+\|\omega\|_{H^2}^4,
\end{equation}
where we used again Proposition~\ref{prop:calderon}-$(ii)$ to estimate $\nabla\psi$ (the omitted multiplicative constant depends on $\Lambda\psi_2$).
Similarly,
 \begin{equation}\begin{split}
\|\nabla(\nabla^\perp\psi\cdot\nabla\omega)\|_{L^2}^2&\leq\|\nabla\nabla^\perp\psi\|_{L^4}^2\|\nabla\omega\|_{L^4}^2+\|\nabla^\perp\psi\|_{L^4}^2\|\nabla^2\omega\|_{L^4}^2\lesssim1+\|\omega\|_{H^3}^4.
     \end{split}
 \end{equation}
 This concludes the estimate of the first term in~\eqref{point de depart}.

 The second term in~\eqref{point de depart} is studied using similar arguments. On the one hand we have, By the Gagliardo-Nirenberg inequality and Proposition~\ref{prop:calderon}-$(ii)$,
 \begin{equation}
\|\nabla\psi\|_{L^\infty}\lesssim\|\nabla\psi\|_{H^2}\lesssim1+\|\omega\|_{H^2},
 \end{equation}
 and on the other hand,
 \begin{equation}\begin{split}
     \|\nabla\partial_t\omega\|_{L^2}=\|\nabla(\nabla^\perp\psi\cdot\nabla\omega)\|_{L^2}&\leq\|\nabla\nabla^\perp\psi\|_{L^4}^2\|\nabla\omega\|_{L^4}^2+\|\nabla^\perp\psi\|_{L^4}^2\|\nabla^2\omega\|_{L^4}^2\lesssim1+\|\omega\|_{H^3}^4.
     \end{split}
 \end{equation}
using again Proposition~\ref{prop:calderon}-$(ii)$. 

Plugging all these estimates back into~\eqref{point de depart} eventually leads to
\begin{equation}
    \|\partial^2_t\omega\|_{L^2} \lesssim1+\|\omega\|_{H^3}^4.
\end{equation}
Another Sobolev embedding gives that $\nabla_{\!t\!,x\,}\omega$ is continuous.

\end{proof}
The obtained regularity is actually slightly better because we have $\partial_t \omega\in\cC^0(0,T;$ $\cC^0(\R^2))$ but we have $\nabla\omega\in\cC^0(0,T:\cC^1(\R^2))$. The property $\omega\in\cC^1([0,T)\times\R^2)$ is important because it means that $\omega$ is a solution in the classical sense.

\subsection{Blow-up criterion}

In our system, it is  unclear whether the existence time is $+\infty$. The blow-up mechanisms may arise either from a purely nonlinear effect in the PDE component (Type I blow-up) or from the collapse of the boundary of the plateau with the boundary of the rigid region ($R(t)\to0$). This leads us to study two types of blow-up and then to derive two different blow-up criterion. These two types of blow-up do not exclude each other, as we can see in the lemma below.

\begin{lemma}[Blow-up criteria]
    Let $\omega$ be the solution given by lemma~\ref{lem:existence} and assume that the time of existence $T^\star$ is finite. Then we have the following blow-up criteron :
    \begin{equation}
        \int_0^{T^\star}\|\omega(t,\cdot)\|_{H^2}\d t =+\infty.
    \end{equation}
\end{lemma}
\noindent
This lemma completes the proof of Theorem~\ref{thrm}-$(ii)$.

\begin{proof}
    In the case where where $R(t)\to0$, this blow-up criterion is a direct consequence of the \textit{a-priori} estimate on $R(t)$ given by Corollary~\ref{coro:a priori R}. If we now assume that $R(t)\geq\mu>0$ for all time, then the \textit{a priori} estimate on the $H^4$ norm of $\omega$ gives
    \begin{equation}
        \frac{d}{dt}\|\omega\|_{H^4}\lesssim\|\omega\|_{H^4}^3+\frac{1}{\mu^5}.
    \end{equation}
    Nevertheless, in our case here we work with a transport equation and it is standard to prove in whole generality (for any transport equation $\partial_t\omega+v\cdot\nabla\omega=0$):
    \begin{equation}
        \|\omega(t,\cdot)\|_{H^s}\leq\|\omega(t,\cdot)\|_{H^s}\exp\bigg(C_s\exp\bigg(C_s\int_0^t\|v(\tau,\cdot)\|_{LL}\d\tau\bigg)\bigg),
    \end{equation}
    where $\|\_\|_{LL}$ designates the log-lipschitz norm.
    This standard estimate is given by a Grönwall-Osgood estimate on the evolution in time of $\|\omega(t,\cdot)\|_{H^s}$. In our case, we have the velocity field that is given by $\nabla^\perp\psi$ and, as a consequence of Proposition~\ref{prop:calderon}-$(ii)$, we have
    \begin{equation}
        \|\nabla^\perp\psi\|_{H^2}\lesssim 1+\|\omega\|_{H^2},
    \end{equation}
    with a constant that depends on $\Lambda\psi_2$. Finally, it is standard~\cite{Bahouri_Chemin_Danchin_2011} to prove that in dimension $2$ the $H^2$ norm controls the log-lipschitz norm:
    \begin{equation}
        \|\nabla^\perp\psi\|_{LL}\lesssim\|\nabla^\perp\psi\|_{H^2}.
    \end{equation}
    This gives the conclusion.
\end{proof}

\subsection{Uniqueness and local stability of the solution}
Let us consider the following system satisfied by $\omega^{i}$, $i=1,2$.
\begin{equation}\label{eq:two}
    \left\{\begin{array}{ll}
    \displaystyle\frac{\partial{\omega}^i}{\partial t}(t,y)+\nabla^\perp{\psi^i}\cdot\nabla{\omega^i}(t,y)=0,\vspace{0.2cm}&\\
         {\psi^{i}_1}(t,y)=0,&\text{for }y\in {\cS},\vspace{0.2cm}\\
\Lambda\displaystyle{\psi^{i}_1}(t,y)={\omega}^{i}(t,y)+\Lambda{\psi_2}(t,y),&\text{for }y\in\R^2\setminus{\cS},\vspace{0.2cm}\\
         {\psi}^{i}(t,y)={\psi_2}(t,y)-{\psi^{i}_1}(t,y)-{\varphi}(t,y),\vspace{0.2cm}\\
         {\varphi}(t,y):=\dot{\theta}(t)\frac{|y|^2}{2}-\Big(\bR_{\theta(t)}y+h(t)\Big)\cdot{\dot{h}(t)}^\perp\vspace{0.2cm}\\
         {\psi_2}(t,y):=\chi(\dist(y,\cS)){\varphi}(t,y).
    \end{array}\right.
\end{equation}
\begin{proposition}
Let $\omega^1$ and $\omega^2$ be two solutions of \eqref{eq:two} with initial data in $H^4(\mathbb{R}^2)$. Then, for $t$ in a sufficiently small time interval $[0,T]$, there exists a constant
\begin{equation}
C = C\big(\|\omega^1(0)\|_{H^4}, \|\omega^2(0)\|_{H^4}\big)
\end{equation}
such that
\begin{equation}
\frac{d}{dt}\|\omega^2 - \omega^1\|_{L^2(\mathcal{F})}
\leq C\,\|\omega^2 - \omega^1\|_{L^2(\mathcal{F})}.
\end{equation}
\end{proposition}
\begin{proof}
We set
\begin{equation}
\delta \omega := \omega^2 - \omega^1,
\qquad
u^i := \nabla^\perp \psi^i.
\end{equation}
We know from \eqref{eq:two}$_1$ that $\omega^i$ satisfies
\begin{equation}
\partial_t \omega^i + u^i \cdot \nabla \omega^i = 0.
\end{equation}
Subtracting the two equations yields
\begin{equation}
\partial_t \delta \omega
+ u^2 \cdot \nabla \delta \omega
= - (u^2 - u^1)\cdot \nabla \omega^1.
\end{equation}

Taking the $L^2$ inner product with $\delta \omega$, we obtain
\begin{equation}
\frac{1}{2}\frac{d}{dt}\|\delta \omega\|_{L^2(\mathcal{F})}^2
= -\int_{\mathcal{F}} (u^2 \cdot \nabla \delta \omega)\,\delta \omega \dx
- \int_{\mathcal{F}} (u^2 - u^1)\cdot \nabla \omega^1 \,\delta \omega \dx.
\end{equation}
Since $u^2 = \nabla^\perp \psi^2$, we have $\nabla \cdot u^2 = 0$. Hence,
\begin{equation}
\int_{\mathcal{F}} (u^2 \cdot \nabla \delta \omega)\,\delta \omega \dx
= \frac{1}{2} \int_{\partial \mathcal{F}} (u^2 \cdot n) \nabla (\delta \omega)^2 \d\sigma.
\end{equation}
Due to the \textit{plateau property} \eqref{eq:plateau}, the boundary term is zero and we obtain
\begin{equation}
\frac{1}{2}\frac{d}{dt}\|\delta \omega\|_{L^2 (\mathcal{F})}^2
= - \int_{\mathcal{F}} (u^2 - u^1)\cdot \nabla \omega^1 \,\delta \omega \dx.
\end{equation}
By H\"{o}lder's inequality,
\begin{equation}
\left|\int_{\mathcal{F}} (u^2 - u^1)\cdot \nabla \omega^1 \,\delta \omega \dx\right|
\leq \|u^2 - u^1\|_{L^2 (\mathcal{F})}
\|\nabla \omega^1\|_{L^\infty(\mathcal{F})}
\|\delta \omega\|_{L^2(\mathcal{F})}.
\end{equation}
Thus,
\begin{equation}\label{ddt1}
\frac{d}{dt}\|\delta \omega\|_{L^2(\mathcal{F})}
\leq \|\nabla \omega^1\|_{L^\infty(\mathcal{F})}\,
\|u^2 - u^1\|_{L^2(\mathcal{F})}.
\end{equation}
Recall that
$\psi^i = \psi_2 - \psi^i_1 - \varphi,
\quad
u^i = \nabla^\perp \psi^i$. Hence,
\begin{equation}\label{diff:u2u1}
u^2 - u^1
= - \nabla^\perp(\psi^2_1 - \psi^1_1).
\end{equation}
We know from equation \eqref{eq:two}$_3$ that
$\Lambda \psi^i_1 = \omega^i + \Lambda \psi_2$, we deduce
\begin{equation}
\Lambda (\psi^2_1 - \psi^1_1)
= \delta \omega.
\end{equation}
Using the boundedness of $\Lambda^{-1}$ from $L^2$ to $ H^1$, we obtain
\begin{equation}\label{L2H1}
\|\nabla (\psi^2_1 - \psi^1_1)\|_{L^2}
\leq C
\|\delta \omega\|_{L^2}.
\end{equation}
Combining the above estimates \eqref{diff:u2u1} and \eqref{L2H1}, we obtain
\begin{equation}\label{ddt2}
\|u^2 - u^1\|_{L^2}
\leq C \|\delta \omega\|_{L^2}.
\end{equation}
Thus, using \eqref{ddt1} and \eqref{ddt2}, we have
\begin{equation}
\frac{d}{dt}\|\delta \omega\|_{L^2}
\leq C \|\nabla \omega^1\|_{L^\infty}
\|\delta \omega\|_{L^2}.
\end{equation}

Since $\omega^1 \in H^4(\cF_0)$, Sobolev embedding yields
\begin{equation}
\|\nabla \omega^1\|_{L^\infty}
\leq C \|\omega^1\|_{H^4}.
\end{equation}

Therefore,
\begin{equation}
\frac{d}{dt}\|\delta \omega\|_{L^2}
\leq C
\|\delta \omega\|_{L^2}.
\end{equation}

\medskip

This concludes the proof.
\end{proof}

\begin{corollary}
    For a given initial value $\omega_0\in H^4$, the solution is unique.
\end{corollary}

\noindent
This lemma completes the proof of Theorem~\ref{thrm}-$(iii)$ and the uniqueness of Theorem~\ref{thrm}-$(i)$. Hence, the theorem in now fully proved.

\appendix
\section{Proofs of the technical lemmas and propositions}

\subsection{Proof of Proposition~\ref{prop:representation}}
The existence of a Green function $G_{\!\cF}$ follows from the representation formulas widely established in the literature. See for instance the most recent advances at~\cite{Abatangelo_2015, Djitte_Sueur_2023_Representation}. We now define the regular part of the Green function as
\begin{equation}\label{Bichet}
    H_{\!\cF}(x,y):=G(x-y)-G_{\!\cF}(x,y).
\end{equation}
Since we have $G$ being the Green function of the fractional Laplace operator $\Lambda$ in the whole plane, we deduce that $H_{\!\cF}$ satisfies 
\begin{equation}
    (-\Delta)^{\frac{1}{2}}H_{\!\cF}(x,\cdot)=0,\quad\text{in }\cF\qquad\text{and}\qquad H_{\!\cF}(x,\cdot)=G(x-\cdot)\quad\text{outside }\cF.
\end{equation}
This ensures that $H_{\!\cF}$ is smooth inside $\cF$ by fractional elliptic regularity~\cite{Sylvestre_2007,Ros-Oton_Serra_2014,Caffarelli_Sylvestre_2009_Regularity, Chen_Kim_Song_2012}. Concerning the regularity up to the boundary~\eqref{eq:estim smooth Green}, it is shown in the standard literature~\cite{Chen_Kim_Song_2012,Ros-Oton_Serra_2014,Abatangelo_2015} that the behavior of the Green function in the neighborhood of the boundary is
\begin{equation}\label{Hurmic}
    G_{\cF}(x,y)\simeq\min\bigg\{\frac{1}{|x-y|};\frac{\sqrt{\dist(x;\partial\cF)\,}\sqrt{\dist(x;\partial\cF)\,}}{|x-y|^2}\bigg\},
\end{equation}
where the notation $\simeq$ means that the two sides are comparable (on a neighborhood of the boundary) up to multiplicative constant. This comparison result is the fractional analogous of the boundary Harnack's principle (well-known for the Dirichlet operator).

In the case where $k=\ell=0$, the estimate~\eqref{Hurmic} gives directly~\eqref{eq:estim smooth Green}, using~\eqref{Bichet} by a straightforward separation of cases depending on whether $|x-y|\leq$ $\sqrt{\dist(x;\partial\cF)\,}$  $\sqrt{\dist(x;\partial\cF)}$ or not. Concerning higher regularity estimates, they rely on the fact that the function $H(\cdot,y)$ is $s$-harmonic. For $u$ a $s$-harmonic function, the non-local Schauder estimates~\cite{Sylvestre_2007,Ros-Oton_Serra_2014,Caffarelli_Sylvestre_2009_Regularity} gives that (for $r>0$ small enough):
\begin{equation}
    |\nabla^k u(x)|\lesssim r^{-k}\sup_{\cB(x,r)}|u|.
\end{equation}
We combine this with the bound obtained when $k=\ell=0$ to conclude the proof of the estimate~\eqref{eq:estim smooth Green}. \qed

\subsection{Proof of Proposition~\ref{prop:calderon}}

\subsubsection{Proof of Proposition~\ref{prop:calderon}-$(i)$}
By construction, the function $G_{\!\delta}$ is $\cC^\infty(\R^2)$ and the function $H_\cF$ is in $\cC^\infty(\cF)$ by $s$-elliptic regularity estimate. Therefore, we have for any $\phi\in L^2$ that 
\begin{equation}
    x\longmapsto\int_{\cF}\Big(G_{\!\delta}(x-y)-H_\cF(x,y)\Big)\,\phi(y)\,\d y\qquad\in\cC^\infty(\cF).
\end{equation}
The prefactor $(1-\chi_\delta(\dist(x,\cS)))$ ensures then that we have  
\begin{equation}
    x\longmapsto\Big(1-\chi_\delta\big(\dist(x,\cS)\big)\Big)\int_{\cF}\Big(G_{\!\delta}(x-y)-H_\cF(x,y)\Big)\,\phi(y)\,\d y
\end{equation}
that is smooth on $\cF$ up to the boundary. Finally, we observe that $x\mapsto\chi_\delta(|x|/\delta^2)$ is supported inside $\cB(0,2/\delta)$, and therefore it is also the case for $G_{\cF\!,\delta}\phi$. The announced estimate then follows from standard manipulations involving smoothing arguments with cut-off functions.

\subsubsection{Proof of Proposition~\ref{prop:calderon}-$(ii)$}
The argument to obtain Calderon-Zygmund type estimate in $L^2$ being quite standard, we do not go into technical details. We mainly present the ingredients of the proof to convince the reader that it works the same in our unbounded case as in classical cases.
In the case of the whole plane, a direct computation using Fourier transform gives that
\begin{equation}
    \widehat{\nabla G\ast\phi}(\xi)=i\frac{\xi}{|\xi|}\widehat{\phi}(\xi),
\end{equation}
so that, by Plancherel identity, $\|\nabla G\ast\phi\|_{L^2}=\|\phi\|_{L^2}.$ In the case of truncated or regularized convolution kernel $G_{\!\delta}$, we observe that we have the so-called Calderon-Zygmund cancellation:
\begin{equation}
    \forall\,r>0,\qquad\int_{|x|=r}\nabla^\perp G_{\!\delta}(x)\,\d\sigma(x)=0.
\end{equation}
This cancellation allows us to manipulate this integral using principal values and then proceed to a Calderon-Zygmund decomposition. This eventually leads to the existence of a constant independent of $\delta$ such that
\begin{equation}
    \big\|\nabla G_{\!\delta}\ast\phi\|\leq C\|\phi\|_{L^2}.
\end{equation}
Concerning the boundary term, meaning the term $x\mapsto\int_{\cF}\nabla H_{\cF}(x,y)\phi(y)\d y$, it is estimated using~\eqref{Hurmic} and straightforward but technical manipulations. 
These manipulations only involve local estimates in the neighborhood of the boundary ; the obtained constant then only depends on the smoothness of this boundary. 
We refer to the literature for details~\cite{Sylvestre_2007,Ros-Oton_Serra_2014,Caffarelli_Sylvestre_2009_Regularity, Chen_Kim_Song_2012, Abatangelo_2015, Djitte_Sueur_2023_Representation}. 
Although our domain $\cF$ is unbounded, its boundary $\partial\cF$ is smooth and bounded so that the same reasoning applies. At long distances from the boundary, the decay of $H_\cF$ is strictly faster than $G$. More precisely, uniformly in $y$, we have as $|x|\to+\infty$:
\begin{equation}
    G_{\cF}(x,y)\sim G_\frac{1}{2}(x-y)=\frac{c}{|x-y|},\qquad{while}\qquad H_\cF(x,y)\lesssim\frac{\sqrt{\dist(x,\partial\cF)\,\dist(y,\partial\cF)\,}}{|x-y|^2}.
\end{equation}
 \qed

\subsubsection{Proof of Proposition~\ref{prop:calderon}-$(iii)$}
Concerning the convergence as $\delta\to0$, we start by studying the convergence of the gradients. We have the convergence
\begin{equation}
    \int_\cF \nabla^\perp G_\delta(x-y)\,\phi(y)\,\d y,
\end{equation}
using standard arguments of principal value theory. The prefactor $\chi_\delta(|x|/\delta^2)(1-\chi_\delta(\dist(x,\cS)))$ also converges to $1$ for all fixed $x\in\cF$. Using the Calderon-Zygmund property stated by Proposition~\ref{prop:calderon}-$(ii)$, we can then pass to the limit in $L^2$ by the Lebesgue dominated convergence theorem. Concerning the convergence with respect to the $L^2$ norm, we observe that
\begin{equation}
    A_\delta(x):=\int_{\cF}\Big( G_\delta(x-y)-G(x-y)\Big)\phi(y)\,\dy = \int_{\cB(x,\delta)}\Big( G_\delta(x-y)-G(x-y)\Big)\phi(y)\,\dy
\end{equation}
Therefore, since $|G_\delta(u)|\leq C|G(u)|$ for all $u$,
\begin{equation}
    |A_\delta(x)|\lesssim\int_{\R^2}G(x-y)\,|\phi(y)|\,\mathbbm{1}_{\cB(x,\delta)}(y)\,\dy
\end{equation}
Using now the Young inequality for convolution, we are led to
\begin{equation}
    |A_\delta|\lesssim\|\phi\|_{L^2}\big\|G\mathbbm{1}_{\cB(0,\delta)}\big\|_{L^1},
\end{equation}
and the convergence for the $L^2$-distance follows by the Lebesgue dominated convergence theorem.\qed

\section*{Acknowledgement} A.R has been partially supported by the Basque Government through the BERC 2022-2025 program and by the Spanish State Research Agency through BCAM Severo Ochoa CEX2021-001142-S and through project PID2023-146764NB-I00 funded by MICIU/AEI/10.13039/501100011033 and cofunded by the European Union. A.R is also supported by the Grant RYC2022-036183-I funded by MICIU/AEI/10.13039/501100011033 and by ESF+.

\bibliographystyle{plain}
\bibliography{bibliography}

\end{document}

%% file: commands.tex
\usepackage[utf8]{inputenc}
\usepackage[T1]{fontenc}
\usepackage[hidelinks]{hyperref}
\usepackage[left = 2.5cm, right = 2.5cm, top = 2.5cm, bottom = 3.cm]{geometry}
\usepackage{bbm}
\usepackage{stmaryrd}
\usepackage{xcolor}
\usepackage{amsmath, amsfonts, amsthm, amssymb}
\usepackage{mathrsfs}
\usepackage{yfonts}

\usepackage{mathtools}
\mathtoolsset{showonlyrefs}

\DeclareMathOperator{\di}{div}

\DeclareMathOperator{\dist}{dist}

\DeclareMathOperator{\supp}{supp}

\renewcommand{\d}{\mathrm{d}}

\DeclareFontFamily{U}{mathx}{}
\DeclareFontShape{U}{mathx}{m}{n}{<-> mathx10}{}
\DeclareSymbolFont{mathx}{U}{mathx}{m}{n}
\DeclareMathAccent{\widehat}{0}{mathx}{"70}
\DeclareMathAccent{\widecheck}{0}{mathx}{"71}

\theoremstyle{plain}
	\newtheorem{theorem}{Theorem}[section]
	\newtheorem{proposition}[theorem]{Proposition}    
	\newtheorem{lemma}[theorem]{Lemma}          
	\newtheorem{corollary}[theorem]{Corollary}

\theoremstyle{definition}
	\newtheorem{definition}[theorem]{Definition}

  \def\bR{\boldsymbol{\mathrm{R}}}

 \def\cB{{\mathcal B}} \def\cC{{\mathcal C}} 
\def\cD{{\mathcal D}}  \def\cF{{\mathcal F}} 
   
  \def\cL{{\mathcal L}} 
 \def\cN{{\mathcal N}}  
   
\def\cS{{\mathcal S}}

 \def\N{{\mathbb N}}  
\def\P{{\mathbb P}}  \def\R{{\mathbb R}} 
   
  \def\X{{\mathbb X}}
